\documentclass[12pt,twoside]{amsart}
\usepackage{amsmath, amsthm, amscd, amsfonts, amssymb, graphicx}
\usepackage[bookmarksnumbered, plainpages]{hyperref}
\usepackage[numbers,sort&compress]{natbib}
\newcommand{\og}{\omega}

\newcommand{\tr}{\text{tr}}

\newcommand{\lag}{\langle}
\newcommand{\rag}{\rangle}

\newtheorem{thm}{Theorem}[section]
\newtheorem{cor}[thm]{Corollary}
\newtheorem{lem}[thm]{Lemma}
\newtheorem{prop}[thm]{Proposition}
\newtheorem{defn}[thm]{Definition}

\newtheorem*{claim}{Claim}
\numberwithin{equation}{section}

\begin{document}

\title{\bf  HERMITIAN-POISSON METRICS ON PROJECTIVELY FLAT COMPLEX VECTOR BUNDLES OVER NON-COMPACT GAUDUCHON MANIFOLDS}
\author{    JIE GENG, ZHENGHAN SHEN AND XI ZHANG}

\thanks{{\scriptsize
\hskip -0.4 true cm \textit{2020 Mathematics Subject Classification:}
53C07; 14J60; 32Q15.
\newline \textit{Key words and phrases:} Projectively flat vector bundles; Gauduchon manifold; Hermitian-Poisson metrics; Non-compact.
\newline \textit The authors are partially supported by the National Natural Science Foundation of China No.12141104 and No.12431004. The first author is also supported by the University Science Research Project of Anhui Province No.2023AH052921.}}
\maketitle

\begin{abstract}
In this paper, we investigate the projectively flat bundles over a class of non-compact Gauduchon manifolds. By combining heat flow techniques and continuity methods, we establish a correspondence between the existence of Hermitian-Poisson metrics and the semi-simplicity property on projectively flat bundles.
\end{abstract}

\vskip 0.2 true cm


\pagestyle{myheadings}
\markboth{\rightline {\scriptsize       }}
         {\leftline{\scriptsize HERMITIAN-POISSON METRICS ON PROJECTIVELY FLAT COMPLEX VECTOR BUNDLES  }}

\bigskip
\bigskip


\section{ Introduction}

Let $M$ be an $n$-dimensional complex manifold equipped with a Hermitian metric $g$ and its associated K\"{a}hler form $\omega$, $(E, D)$ be a complex vector bundle of rank $r$ over $M$. A vector bundle $(E, D)$ is said to be simple if it has no proper $D$-invariant sub-bundle, and it is called semi-simple if it is a direct sum of simple vector bundles. Given a Hermitian metric $H$ on $E$, there exists a unique decomposition
\begin{equation}
D = D_H + \psi_H,
\end{equation}
where \( D_H \) is a unitary connection and \( \psi_H \in \Omega^1(\text{End}(E)) \) is self-adjoint with respect to \( H \). A Hermitian metric \( H \) on \( (E, D) \) is said to be harmonic if it is a critical point of the energy functional $\frac{1}{2}\int_M |\psi_{H}|^{2}_{H,g} \frac{\omega^n}{n!}$, i.e., it satisfies the  following Euler-Lagrange equation:
\begin{equation}
D_H^* \psi_H = 0,
\end{equation}
where \( D_H^* \) denotes the formal adjoint of \( D_H \). Under the assumption that the flat complex vector bundle $(E, D)$ is semi-simple over a compact K\"{a}hler manifold $(M, g)$, Corlette \cite{kc1} and Donaldson \cite{skd} proved the existence of harmonic metrics. This result implies that there exists a poly-stable Higgs structure on \( E \). On the other hand, the Hitchin-Kobayashi correspondence  provides a fundamental link between the stability of holomorphic vector bundles and the existence of the Hermitian-Einstein metrics.  It was constructed through foundational works by Narasimhan-Seshadri \cite{NM11} on compact Riemann surfaces, Donaldson \cite{DON85} on algebraic manifolds, Uhlenbeck-Yau \cite{UY86} on higher-dimensional compact K\"{a}hler manifolds, and Hitchin \cite{HiNJ} and Simpson \cite{SIM} on Higgs bundles. Based on the work of Corlette \cite{kc1}, Donaldson \cite{skd}, Hitchin \cite{HiNJ} and Simpson \cite{SIM,SIMMMRE}, one obtains the classical non-abelian Hodge correspondence. It establishes a one-to-one relationship between the category of semi-simple flat bundles and the category of poly-stable Higgs bundles with vanishing Chern classes on compact K\"{a}hler manifolds. This correspondence has inspired numerous significant generalizations and extensions, including Jost-Zuo \cite{JZ1996,jjkz}, Biquard-Boalch \cite{bobp} and Mochizuki \cite{Mo2006,tm1} on quasiprojective varieties, L\"{u}bke-Teleman \cite{LT95,LT} on non-K\"{a}hler surfaces, Biswas-Kasuya \cite{mpoch98} on compact Sasakian manifolds, and Pan-Shen-Zhang \cite{ppczhxz} on non-Hermitian Yang-Mills
bundles. Other related studies refer to \cite{mpoch98,Bradlow,Mundet,Gothen,Mm3330,LiJY1996,xz2021}.

In order to extend the non-abelian Hodge correspondence to non-K\"{a}hler manifolds, Pan-Zhang-Zhang \cite{cppcjzz} introduced the notion of Hermitian-Poisson metrics on projectively
flat bundles. Define the operators:
\begin{equation}
D_H'' = \bar{\partial}_H + \psi_H^{1,0}, \quad D_H' = \partial_H + \psi_H^{0,1},
\end{equation}
and
\begin{equation}
G_H = (D_H'')^2 = \bar{\partial}_H^2 + \bar{\partial}_H \psi_H^{1,0} + \psi_H^{1,0} \wedge \psi_H^{1,0},
\end{equation}
where \( \partial_H \) (resp. \( \bar{\partial}_H \)) denotes the \( (1,0) \)-part (resp. \( (0,1) \)-part) of \( D_H \). We say the Hermitian metric \( H \) is called a Hermitian-Poisson metric on \( (E, D) \) if it satisfies
\begin{equation}
\sqrt{-1} \Lambda_\omega G_H = \lambda \cdot \text{Id}_E,
\end{equation}
where \( \Lambda_\omega \) denotes the contraction with \( \omega \), and \( \lambda \) is a real constant. In \cite{cppcjzz}, an existence theorem was established for Hermitian-Poisson metrics on simple projectively flat bundles over compact Gauduchon manifolds. In this paper, we extend the study to the non-compact Gauduchon manifolds case. Here, a Hermitian metric \( g \) is called Gauduchon if its K\"{a}hler form $\omega$ satisfies \( \partial \bar{\partial} \omega^{n-1} = 0 \). Gauduchon \cite{pggg} proved that within the conformal class of every Hermitian metric \( g \), there exists a Gauduchon metric.

A connection \( D \) on the complex vector bundle \( E \) is termed projectively flat if it satisfies the induced connection on the principal $\mathrm{PGL}(r, \mathbb{C})$-bundle is flat, or equivalently, its curvature \( F_D \) is a scalar multiple of the identity endomorphism on vector bundle \( E \). The vector bundle \( E \) is called projectively flat if it admits such a projectively flat connection.

In this work, we focus on the case that the base manifold \( (M, g) \) is non-compact Gauduchon and satisfies the following assumptions.

    \textbf{Assumption 1.} The manifold \( (M, g) \) has finite volume.

     \textbf{Assumption 2.} There exists a non-negative exhaustion function \( \phi \) with \( \sqrt{-1} \Lambda_\omega \partial \bar{\partial} \phi \) bounded.

     \textbf{Assumption 3.} There exists an increasing function \( b: [0, +\infty) \rightarrow [0, +\infty) \) satisfying \( b(0) = 0 \) and \( b(x) = x \) for \( x > 1 \), such that if \( f \) is a bounded positive function on \( M \) with \( \sqrt{-1} \Lambda_\omega \partial \bar{\partial} f \geq -A \), then
    \[
    \sup_M |f| \leq C(A) \, b\left( \int_M |f| \frac{\omega^n}{n!} \right).
    \]
    where $A$ is a positive constant. Moreover, \( \sqrt{-1} \Lambda_\omega \partial \bar{\partial} f \geq 0 \) implies \( \sqrt{-1} \Lambda_\omega \partial \bar{\partial} f = 0 \).

In this paper, we establish a correspondence between the existence of Hermitian-Poisson metrics and the semi-simplicity property for projectively flat bundles over a class of non-compact Gauduchon manifolds under the above assumptions. Our main result is as follows.
\begin{thm} \label{theorem1}
Suppose that $(M, g)$ is a non-compact Gauduchon manifold satisfying Assumptions 1, 2, and 3, with the condition $|\mathrm{d}\omega^{n-1}|_{g} \in L^2(M)$. Let $(E, D)$ be a projectively flat vector bundle over $M$ equipped with a background Hermitian metric $H_0$ such that $\sup\limits_M |\Lambda_{\omega} G_{H_0}|_{H_0} < +\infty$. If $(E, D)$ is simple, then there exists a Hermitian metric $H$ satisfying $D (\log H_0^{-1}H) \in L^2$, with $H$ and $H_0$ mutually bounded, such that
\begin{equation*}
\sqrt{-1}\Lambda_{\omega} G_H = \lambda_{H_0, \omega} \cdot \mathrm{Id}_E,
\end{equation*}
where the constant $\lambda_{H_0, \omega}$ is given by
$$\lambda_{H_0, \omega} = \frac{\int_{M}\sqrt{-1}{\rm tr}(\Lambda_{\omega} G_{{H_{0}}})\frac{\omega^{n}}{n!}}{\mathrm{rank}(E) \, \mathrm{Vol}(M, g)}.$$
Moreover, if $\psi_{{H_{0}}} \in L^2$, then $\psi_H \in L^2$.
Conversely, if $\omega$ is a balanced metric and $(E, D)$ admits a Hermitian-Poisson metric $H$ with $\psi_H \in L^2$, then $(E, D)$ is semi-simple.
\end{thm}

In \cite{SIM}, Simpson utilized Donaldson's heat flow technique to establish the existence of Hermitian-Einstein metrics on Higgs bundles, and his argument depends crucially on the Donaldson functional's properties. However, this approach is not viable when the metric $g$ is merely Gauduchon, as the Donaldson functional is not well-defined in this case. Therefore, Simpson's argument does not directly translate to our situation. In this paper, we employ the continuity method of Uhlenbeck-Yau \cite{UY86}, which provides a natural framework for our analysis. We first consider the following perturbed equation on $(M, g)$:
\begin{equation} \label{eq}
\sqrt{-1}\Lambda_{\omega } G_{H}-\lambda \cdot \textmd{Id}_E-\varepsilon \log (H_{0}^{-1}H)=0, \varepsilon \in (0,1].
\end{equation}
The solvability of this perturbed equation is guaranteed by the Fredholm property of elliptic operators on compact manifolds. However, the property fails to hold for non-compact manifolds, which means we cannot directly use this method to solve the perturbation equations in such cases. To address this limitation, we adapt an approach combining heat flow and exhaustion techniques, which was initially introduced by Zhang-Zhang-Zhang \cite{cjzpzz} in their study of Hermitian-Einstein metrics on non-compact manifolds. Building upon their work, we extend this methodology to solve the above perturbed equation in our setting. Detailed proofs are provided in Sections 3 and 4. For simplicity, we define
\begin{equation}\label{wwwww2} \Phi(H)=\sqrt{-1}\Lambda_{\omega } G_{H}-\lambda_{H_{0}, \omega } \cdot \textmd{Id}_E-\varepsilon \log (H_{0}^{-1}H). \end{equation}
Under the conditions specified in Theorem \ref{theorem1}, we derive the following identity:
\begin{equation}\label{eq33}
4 \int_M \tr(\Phi(H) s)\frac{\og^n}{n!}+\int_M\lag\Psi(s)(Ds), Ds\rag_{H_{0}}\frac{\og^n}{n!}+4 \varepsilon \|s\|^{2}_{L^2} = 4 \int_M \tr(\Phi(H_{0}) s)\frac{\og^n}{n!},
\end{equation}
where $s=\log (H_{0}^{-1}H)$ and
\begin{equation}\label{eq3301}
\Psi(x,y)=
\begin{cases}
&\frac{e^{y-x}-1}{y-x},\ \ \ x\neq y;\\
&\ \ \ \  1,\ \ \ \ \ \  x=y.
\end{cases}
\end{equation}
Using the above identity (\ref{eq33}) and the result of Pan-Shen-Zhang {\cite{ppczhxz}}, we can establish the existence of the Hermitian-Poisson metric via the continuity method. It should be noted that the identity (\ref{eq33}) replaces the Donaldson functional, which is crucial to our argument (see Section 5).

This paper is organized as follows. In Section 2, we introduce the necessary preliminaries, derive key estimates required for later proofs, and establish the identity (\ref{eq33}). In Section 3, we prove the long-time existence of solutions to the heat flow equation (\ref{Flow}) and establish the solvability of the perturbed equation (\ref{eq}) on compact Gauduchon manifolds.
In Section 4, we generalize these results to non-compact Gauduchon manifolds. Finally, in Section 5, we present the proof of Theorem \ref{theorem1}.

\section{Preliminary results}

Let $(M, g)$ be a Hermitian manifold of dimension $n$, and $(E,D)$ be
a complex vector bundle over $M$ with rank $r$.
For any Hermitian metric $H$ on $E$, there exists a unique decomposition
\begin{equation}\label{hgfdgcxz90}
D=D_{H}+\psi_{H},
\end{equation}
where $D_{H}$ is a unitary connection and $\psi_{H}\in \Omega^{1}(\mbox{End}(E))$ is self-adjoint with respect to $H$, satisfying
\begin{equation}
H(\psi_{H}s_{1},s_{2}) = \frac{1}{2}\{H(Ds_{1},s_{2}) + H(s_{1},Ds_{2}) - dH(s_{1},s_{2})\}, \ \text{\small for all } s_{1},s_{2} \in \Gamma(E).
\end{equation}

Suppose that \( E \) is a projectively flat vector bundle over a compact complex manifold \( M \) with rank \( r \), and its first Chern class satisfies \( c_1(E) \cap \Omega^{1,1}(M, \mathbb{C}) \cap \Omega^2(M, \mathbb{R}) \neq \emptyset \). Then, there exists a connection \( D \) on \( E \) such that its curvature $F_D$ satisfies
\begin{equation}\label{condition233}
\sqrt{-1} F_D = \alpha \otimes \mathrm{Id}_E,
\end{equation}
where \( \alpha \in \frac{2\pi}{r} c_1(E) \) is a real \( (1,1) \)-form. Expanding $F_{D}$ via the decomposition (\ref{hgfdgcxz90}), we rewrite (\ref{condition233}) as
\begin{equation}\label{condition3}
\sqrt{-1}F_{D}=\sqrt{-1}(D_{H}^{2}+D_{H}\circ \psi_{H}+ \psi_{H}\circ D_{H}+\psi_{H}\wedge \psi_{H})=\alpha \otimes \mathrm{Id}_{E}.
\end{equation}
By decomposing the above identity (\ref{condition3}) into the self-adjoint and anti-self-adjoint components, we obtain
\begin{equation}\label{co1}
D_{H}(\psi_{H})=0,
\end{equation}
and
\begin{equation}\label{coyyy245}
\sqrt{-1}(D_{H}^{2}+\psi_{H}\wedge \psi_{H})=\alpha \otimes \mathrm{Id}_{E}.
\end{equation}
Furthermore, if we consider the decompositions of (\ref{co1})and (\ref{coyyy245}) into $(1,0)$-part and $(0,1)$-part, we obtain the following equations:
\begin{equation}\label{dd:1}
\left\{\begin{split}
&\bar{\partial}_{H}^{2}+\psi_{H}^{0,1}\wedge\psi_{H}^{0,1}=0, \quad \partial_{H}^{2}+\psi_{H}^{1, 0}\wedge\psi_{H}^{1,0}=0,\\
&\bar{\partial}_{H}\psi_{H}^{1,0}+\partial_{H}\psi_{H}^{0,1}=0,\\
&\partial_{H}\psi_{H}^{1,0}=0, \quad \bar{\partial}_{H}\psi_{H}^{0,1}=0, \\
&\sqrt{-1}([\partial_{H},\bar{\partial}_{H}]+[\psi_{H}^{1,0},\psi_{H}^{0,1}])=\alpha \otimes \mathrm{Id}_{E}.
\end{split}\right.
\end{equation}

Suppose $H_{0}$ is another Hermitian metric on $E$. Set $h=H_{0}^{-1}H$ and $D_{H_{0}}^{c}=D_{H_{0}}^{''}-D_{H_{0}}^{'}$. By direct computation, the following relationships are established:
\begin{equation}
D_{H}^{''}-D_{H_{0}}^{''}=\frac{1}{2}h^{-1}D_{H_{0}}^{c}(h),
\end{equation}
\begin{equation}
D^{c}_{H}-D^{c}_{H_{0}}=h^{-1}D_{H_{0}}^{c}(h),
\end{equation}
and
\begin{equation}\label{GHK01}
G_{H}-G_{H_{0}}=\frac{1}{4}D(h^{-1}D_{H_{0}}^{c}(h)).
\end{equation}

In this paper, we consider the following heat flow on $E$ with the initial Hermitian metric $H_{0}$:
\begin{equation} \label{Flow}
H^{-1}(t)\frac{\partial H(t)}{\partial t}=4\left( \sqrt{-1}\Lambda_{\omega } G_{H(t)}-\lambda \cdot \textmd{Id}_E-\varepsilon \log (H_0^{-1}H(t)) \right),
\end{equation}
where $H(t)$ denotes a family of Hermitian metrics on $E$ and $\varepsilon \in (0,1]$ is a constant.
Choosing a local complex coordinates $\{ z^{\alpha}\}_{\alpha=1}^{n}$ on $M$, the K\"{a}hler form $\omega$ can be expressed as $\omega =\sqrt{-1} g_{\alpha \bar{\beta }} dz^{\alpha} \wedge d\overline{z}^{\beta}$.  We
define the complex Laplace operator for functions as
$$
\widetilde{\Delta}f=-2\sqrt{-1}\Lambda_{\omega } \bar{\partial }\partial
f =\sqrt{-1}\Lambda_{\omega }dd^{c}f=2g^{\alpha \bar{\beta }}\frac{\partial ^{2 }f}{\partial z^{\alpha }\partial
\bar{z}^{\beta }},
$$
where $d^{c}=\bar{\partial }-\partial$ and $(g^{\alpha \bar{\beta }})$ is the inverse matrix of $(g_{\alpha \bar{\beta }})$.
As usual, the Beltrami-Laplacian operator is denoted by $\Delta$. The difference between these two Laplacians is given via a first-order differential operator:
\begin{equation*} \label{laplacian}
(\widetilde{\Delta}-\Delta)f=\langle V,\nabla f\rangle_g,
\end{equation*}
where $V$ is a well-defined vector field on $M$.
\begin{prop}  \label{P1}
Let $H(t)$ be a solution to the heat flow \eqref{Flow}, then we have
\begin{equation}\label{trace}
(\frac{\partial}{\partial t}-\widetilde{\Delta})\{e^{4\varepsilon t}\cdot \mathrm{tr} (\sqrt{-1}\Lambda_{\omega }G_{H(t)}-\lambda \cdot \mathrm{Id}_E-\varepsilon \log h(t))\}=0
\end{equation}
and
\begin{equation} \label{dec}
(\frac{\partial}{\partial t}-\widetilde{\Delta})|\sqrt{-1}\Lambda_{\omega }G_{H(t)}-\lambda \cdot \mathrm{Id}_E-\varepsilon \log h(t)|^2_{H(t)}\leq 0,
\end{equation}
\end{prop}
where $h(t)= H_{0} ^{-1}H(t)$ and $\mathrm{tr}$ denotes the trace operator.
\begin{proof}
For simplicity, set $\Phi_{\varepsilon} (H(t))= \sqrt{-1}\Lambda_{\omega }G_{H(t)}-\lambda \cdot \textmd{Id}_E-\varepsilon \log h(t)$. It follows that \(\Phi_{\varepsilon}(H(t))\) is self-adjoint with respect to \(H(t)\), i.e., \(\Phi_{\varepsilon}(H(t))^{\ast H} = \Phi_{\varepsilon}(H(t))\). By direct computation, we derive
\begin{equation}\label{trace1}
\begin{split}
\frac{\partial}{\partial t}\Phi_{\varepsilon}(H(t))&=\frac{\partial}{\partial t} \{\Phi_{\varepsilon}(H_{0})+\sqrt{-1}\Lambda_{\omega }(G_{H(t)}-G_{H_{0}})-\varepsilon \log h(t)\}\\
&=\frac{\sqrt{-1}}{4}\frac{\partial}{\partial t} \{\Lambda_{\omega}D(h^{-1}D_{H_{0}}^{c}h)\}-\varepsilon \frac{\partial}{\partial t} (\log h(t)) \\
&=\frac{\sqrt{-1}}{4} \Lambda_{\omega} D \{ D_{H_{0}}^{c}(h^{-1}\frac{\partial h}{\partial t})+[h^{-1}D_{H_{0}}^{c}h,h^{-1}\frac{\partial h}{\partial t}]\}-\varepsilon \frac{\partial}{\partial t} (\log h(t)) \\
&=\frac{\sqrt{-1}}{4}\Lambda_{\omega } D (D_{H}^{c}(h^{-1}\frac{\partial h}{\partial t}))-\varepsilon \frac{\partial}{\partial t}(\log h(t)),
\end{split}
\end{equation}
and
\begin{equation*}
\begin{split}
\widetilde{\Delta}|\Phi_{\varepsilon}(H(t))|^2_{H(t)}&=\sqrt{-1}\Lambda_{\omega } dd^{c} \textmd{tr}\{\Phi_{\varepsilon}(H(t))\circ\ \Phi_{\varepsilon}(H(t))\}\\
&=\sqrt{-1}\Lambda_{\omega }\textmd{tr} DD_{H}^{c} \{\Phi_{\varepsilon}(H(t))\circ\ \Phi_{\varepsilon}(H(t))\}\\
&=2\sqrt{-1}\Lambda_{\omega} \textmd{tr}\{DD_{H}^{c}\Phi_{\varepsilon}(H(t))\circ\ \Phi_{\varepsilon}(H(t))\}+2|D\Phi_{\varepsilon}(H(t))|^2_{H(t),\omega},
\end{split}
\end{equation*}
where the last equality follows from the following claim.
\begin{claim}
Let $H$ be a Hermitian metric on $E$, set $$\mathrm{Herm}(E,H)=\{\zeta\in \Gamma(M,\mathrm{End}(E))| \zeta^{*H}=\zeta\}.$$
For any $\zeta\in\mathrm{Herm}(E,H)$ it holds
\begin{equation}\label{WWWPPP2}
\sqrt{-1}\Lambda_{\omega}\mathrm{tr}(D\zeta\wedge D_{H}^{c}\zeta)=|D\zeta|^2_{H,\omega}.
\end{equation}
\end{claim}

\begin{proof}
According to \cite[P.237]{LT95}, for any $\zeta \in \mathrm{Herm}(E,H)$, there exists an open dense subset $V \subset M$ such that at every point $x \in V$, one can find a neighborhood $U$ of $x$, a local $H$-unitary frame $\{e_i\}_{i=1}^r$, and functions $\lambda_i \in C^\infty(U,\mathbb{R})$ satisfying
\begin{eqnarray*}
\zeta(y)=\sum_{i=1}^r \lambda_i(y)\cdot e_i(y)\otimes e^i(y)
\end{eqnarray*}
for all $y\in U$, where $\{e^i\}_{i=1}^{r}$ denotes the dual frame of $E^*$.
Let $D = d + A$, where $D_H = d + \frac{1}{2}(A - \overline{A}^T)$ and $\psi_H = \frac{1}{2}(A + \overline{A}^T)$.
Let $B$ and $C$ be the $(1,0)$ and $(0,1)$ parts of $A$, defined by $B(e_i) = B_i^j e_j$ and $C(e_i) = C_i^j e_j$.
This implies
\begin{align*}
D\zeta = \sum_{i=1}^{r} d\lambda_{i} e_{i} \otimes e^{i}
+ \sum_{i \neq j} (\lambda_{i} - \lambda_{j})(B_{i}^{j} + C_{i}^{j}) e_{j} \otimes e^{i}.\\
D_{H}^{c}\zeta = \sum_{i=1}^{r} d^{c}\lambda_{i} \, e_{i} \otimes e^{i}
+ \sum_{i \neq j} (\lambda_{i} - \lambda_{j})(\overline{C_{j}^{i}} - \overline{B_{j}^{i}}) \, e_{j} \otimes e^{i}.
\end{align*}
Then
\begin{equation}
\begin{split}
&\sqrt{-1}\Lambda_{\omega}\mathrm{tr}(D\zeta\wedge D_{H}^{c}\zeta)\\
=&\sum_{i=1}^{r}\sqrt{-1}\Lambda_{\omega}(d\lambda_{i}\wedge d^{c}\lambda_{i})
+\sum_{i\neq j}(\lambda_{i}-\lambda_{j})^{2}\sqrt{-1}\Lambda_{\omega}(B^{j}_{i}+C_{i}^{j})\wedge (\overline{B_{i}^{j}}-\overline{C_{i}^{j}})\\
=&\sum_{i=1}^{r}|d\lambda_{i}|^{2}+\sum_{i\neq j}(\lambda_{i}-\lambda_{j})^{2}(|B_{i}^{j}|^{2}+|C_{i}^{j}|^{2})\\
=&|D\zeta|^2_{H,\omega}.
\end{split}
\end{equation}
\end{proof}

By applying (\ref{trace1}), we can verify that
\begin{equation}\label{trace2}
(\frac{\partial}{\partial t}-\widetilde{\Delta}) \tr \Phi_{\varepsilon}(H(t))=-4\varepsilon \tr \Phi_{\varepsilon}(H(t)).
\end{equation}
This directly implies (\ref{trace}).

From
\cite[Proposition 2.1]{cjzpzz}, we know that
\begin{equation*}
\langle \frac{\partial}{\partial t}(\log h), h^{-1}\frac{\partial h}{\partial t}\rangle_{H(t)}\geq 0.
\end{equation*}
Combining \eqref{trace1} and \eqref{trace2}, we obtain
\begin{equation*}\label{cur1}
\begin{split}
(\frac{\partial}{\partial t}-\widetilde{\Delta})|\Phi_{\varepsilon}(H(t))|^2_{H(t)}
&=-2|D\Phi_{\varepsilon}(H(t))|^2_{H(t),\omega}-2\varepsilon \langle \frac{\partial}{\partial t}(\log h),\Phi_{\varepsilon}(H(t))\rangle_{H(t)}\\
&\leq 0.
\end{split}
\end{equation*}
This completes the proof of (\ref{dec}).
\end{proof}

In \cite{DON85}, Donaldson introduced the Donaldson's distance on the space of Hermitian metrics for vector bundles.
\begin{defn}
For any Hermitian metrics $H$, $K$ on the vector bundle $E$, the Donaldson's distance between $H$ and $K$ is given by
$$\sigma(H,K)=\mathrm{tr}(H^{-1}K)+\mathrm{tr}(K^{-1}H)-2\mathrm{rank}(E).$$
\end{defn}

It can be shown that $\sigma(H,K)\geq 0$, with equality holding if and only
if $H=K$. A sequence of metrics $\{H_{i}\}$ converges in the $C^{0}$-topology
to $H$ if and only if $\underset{M}{\sup}\ \sigma(H_{i}, H)
 \to 0$ as $i \to \infty$.
\begin{prop}  \label{P243}
Let $H(t)$, $K(t)$ be two solutions to the heat flow \eqref{Flow} with the same initial Hermitian metric $H_{0}$, it holds
$$(\frac{\partial}{\partial t}-\widetilde{\Delta})\sigma(H(t),K(t))\leq 0.$$
\end{prop}

\begin{proof}
Set $h(t)=K(t)^{-1}H(t)$. The equation \eqref{GHK01} gives
\begin{equation}\label{WANMW12}
h(\sqrt{-1}\Lambda_{\omega }(G_{H}-G_{K}))=\frac{1}{4}\sqrt{-1}\Lambda_{\omega }DD_{K}^{c}h-\frac{1}{4}\sqrt{-1}\Lambda_{\omega }Dh\wedge h^{-1}D_{K}^{c}h,
\end{equation}
and
\begin{equation}\label{WANMW213}
h^{-1}(\sqrt{-1}\Lambda_{\omega }(G_{K}-G_{H}))=\frac{1}{4}\sqrt{-1}\Lambda_{\omega }DD_{H}^{c}h^{-1}-\frac{1}{4}\sqrt{-1}\Lambda_{\omega }Dh^{-1}\wedge hD_{H}^{c}h^{-1}.
\end{equation}
Taking the trace on both sides of \eqref{WANMW12} and \eqref{WANMW213}, we have
\begin{equation}\label{GKJHGK541}
\widetilde{\Delta}\textmd{tr}h
=4\textmd{tr}\{h(\sqrt{-1}\Lambda_{\omega }(G_{H}-G_{K}))\}+\sqrt{-1}\Lambda_{\omega }\textmd{tr}(Dh\wedge h^{-1}D_{K}^{c}h),
\end{equation}
and
\begin{equation}\label{GKJHGK5423}
\widetilde{\Delta}\textmd{tr}h^{-1}
=4\textmd{tr}\{h^{-1}(\sqrt{-1}\Lambda_{\omega }(G_{K}-G_{H}))\}+\sqrt{-1}\Lambda_{\omega }\textmd{tr}(Dh^{-1}\wedge hD_{H}^{c}h^{-1}).
\end{equation}
Combining the equalities \eqref{GKJHGK541} and \eqref{GKJHGK5423}, we obain
\begin{equation*}
\begin{split}
&\widetilde{\Delta}(\textmd{tr}h+\textmd{tr}h^{-1})\\
=&4\textmd{tr}\{h(\sqrt{-1}\Lambda_{\omega }(G_{H}-G_{K}))\}+4\textmd{tr}\{h^{-1}(\sqrt{-1}\Lambda_{\omega }(G_{K}-G_{H}))\}\\
&+\sqrt{-1}\Lambda_{\omega }\textmd{tr}(Dh\wedge h^{-1}D_{K}^{c}h)+\sqrt{-1}\Lambda_{\omega }\textmd{tr}(Dh^{-1}\wedge hD_{H}^{c}h^{-1}).
\end{split}
\end{equation*}
By using \eqref{Flow}, we have
\begin{equation*}
\begin{split}
&\frac{\partial}{\partial t}(\textmd{tr}h+\textmd{tr}h^{-1})\\
=&\textmd{tr}(-K^{-1}\frac{\partial K}{\partial t}K^{-1}H+K^{-1}\frac{\partial H}{\partial t})+\textmd{tr}(-H^{-1}\frac{\partial H}{\partial t}H^{-1}K+H^{-1}\frac{\partial K}{\partial t})\\
=&\textmd{tr}\{h(4\sqrt{-1}\Lambda_{\omega }(G_{H}-G_{K})+4\varepsilon \log(H_0^{-1}K)-4\varepsilon \log(H_0^{-1}H))\}\\
&+\textmd{tr}\{h^{-1}(4\sqrt{-1}\Lambda_{\omega }(G_{K}-G_{H})+4\varepsilon \log(H_0^{-1}H)-4\varepsilon \log(H_0^{-1}K))\}\\
=&4\textmd{tr}\{h(\sqrt{-1}\Lambda_{\omega }(G_{H}-G_{K}))+h^{-1}(\sqrt{-1}\Lambda_{\omega }(G_{K}-G_{H}))\}\\
&+4\varepsilon \textmd{tr} \{h (\log (H_0^{-1}K)-\log (H_0^{-1}H)) + h^{-1}(\log (H_0^{-1}H)-\log (H_0^{-1}K))\}.
\end{split}
\end{equation*}
Then, we obtain
\begin{equation*}
\begin{split}
&(\frac{\partial}{\partial t}-\widetilde{\Delta})(\textmd{tr}h+\textmd{tr}h^{-1})\\
=&-\textmd{tr}(\sqrt{-1}\Lambda_{\omega }Dh\wedge h^{-1}D_{K}^{c}h)-\textmd{tr}(\sqrt{-1}\Lambda_{\omega }Dh^{-1}\wedge hD_{H}^{c}h^{-1})\\
&~~+4\varepsilon \textmd{tr} \{h (\log (H_0^{-1}K)-\log (H_0^{-1}H)) + h^{-1}(\log (H_0^{-1}H)-\log (H_0^{-1}K))\}\\
=&-|Dh\cdot h^{-1/2}|_{K,\omega}^{2}-|Dh^{-1}\cdot h^{1/2}|_{H,\omega}^{2}+4\varepsilon \textmd{tr} \{h (\log (H_0^{-1}K)-\log (H_0^{-1}H))\\
& + h^{-1}(\log (H_0^{-1}H)-\log (H_0^{-1}K))\}\\
\leq & 0,
\end{split}
\end{equation*}
where the second equality holds by
\begin{equation*}
\textmd{tr}(\sqrt{-1}\Lambda_{\omega }Dh\wedge h^{-1}D_{K}^{c}h)=|Dh\cdot h^{-1/2}|_{K,\omega}^{2}
\end{equation*}
and
\begin{equation*}
\textmd{tr}(\sqrt{-1}\Lambda_{\omega }Dh^{-1}\wedge hD_{H}^{c}h^{-1})=|Dh^{-1}\cdot h^{1/2}|_{H,\omega}^{2},
\end{equation*}
the last inequality follows from the computations in \cite[Proposition 2.3]{cjzpzz}, which yield
\begin{equation*}
4\varepsilon \textmd{tr} \{h (\log (H_0^{-1}K)-\log (H_0^{-1}H))+ h^{-1}(\log (H_0^{-1}H)-\log (H_0^{-1}K))\} \leq 0.
\end{equation*}
This completes the proof.
\end{proof}

\begin{cor} \label{uniq}
Let $H$, $K$ be two Hermitian metrics satisfying \eqref{eq}, then we have
$$\widetilde{\Delta}\sigma(H,K)\geq 0.$$
\end{cor}

\begin{prop}  \label{A2111}
Let $H(t)$ be a solution to the heat flow \eqref{Flow} with the initial Hermitian metric $H_{0}$, it holds
$$(\widetilde{\Delta}-\frac{\partial}{\partial t}) \log \{ {\rm tr} (H_0^{-1}H)+ {\rm tr} (H^{-1}H_0)\} \geq -4|\sqrt{-1}\Lambda_{\omega } G_{H_{0}}-\lambda \cdot \mathrm{Id}_E-\varepsilon \log (H_0^{-1}H)|_{H_{0}}.$$
\end{prop}

\begin{proof}
Set $h(t)=H_{0}^{-1}H(t)$. Then we have
\begin{equation*}
\begin{split}
&(\widetilde{\Delta}-\frac{\partial}{\partial t})(\textmd{tr}h+\textmd{tr}h^{-1})\\
=&-4\textmd{tr}\{h(\sqrt{-1}\Lambda_{\omega }G_{H_{0}}-\lambda\cdot \textmd{Id}_{E}-\varepsilon \log h)\}+4\textmd{tr}\{h^{-1}(\sqrt{-1}\Lambda_{\omega }G_{H_{0}}-\lambda\cdot \textmd{Id}_{E}\\
&-\varepsilon \log h)\}+\sqrt{-1}\Lambda_{\omega }\textmd{tr}(Dh\wedge h^{-1}D_{H_{0}}^{c}h)+\sqrt{-1}\Lambda_{\omega }\textmd{tr}(Dh^{-1}\wedge hD_{H}^{c}h^{-1}).\\
\end{split}
\end{equation*}
By a similar argument as in \cite{ytslt}, it can be verified that
\begin{equation*}
 \frac{1}{\textmd{tr}h}\cdot\ \textmd{tr}(\sqrt{-1}\Lambda_{\omega }Dh\wedge h^{-1}D_{H_{0}}^{c}h)-\frac{1}{(\textmd{tr}h)^{2}} |d\ {\rm tr} h|^{2} \geq 0
\end{equation*}
and
\begin{equation*}
 \frac{1}{\textmd{tr}h^{-1}}\cdot\ \textmd{tr}(\sqrt{-1}\Lambda_{\omega }Dh^{-1}\wedge hD_{H}^{c}h^{-1})-\frac{1}{(\textmd{tr}h^{-1})^{2}} |d\ {\rm tr} h^{-1}|^{2} \geq 0.
\end{equation*}
From the above two inequalities, we derive
\begin{equation*}
\begin{split}
\frac{1}{\textmd{tr}h+\textmd{tr}h^{-1}}\cdot \textmd{tr} \{\sqrt{-1}\Lambda_{\omega }Dh\wedge h^{-1}D_{H_{0}}^{c}h+\sqrt{-1}\Lambda_{\omega }Dh^{-1}\wedge hD_{H}^{c}h^{-1} \}
\geq \frac{|d{\rm tr} h+ d{\rm tr} h^{-1}|^{2}}{({\rm tr} h+ {\rm tr} h^{-1})^{2}}.
\end{split}
\end{equation*}
Then, we have
\begin{equation*}
\begin{split}
&(\widetilde{\Delta}-\frac{\partial}{\partial t}) \log( {\rm tr} h+ {\rm tr} h^{-1})\\
=&\frac{1}{\textmd{tr}h+\textmd{tr}h^{-1}}\cdot(\widetilde{\Delta}-\frac{\partial}{\partial t}) ( {\rm tr} h+ {\rm tr} h^{-1})-\frac{|d{\rm tr} h+ d{\rm tr} h^{-1}|^{2}}{({\rm tr} h+ {\rm tr} h^{-1})^{2}}\\
=&-\frac{|d{\rm tr} h+ d{\rm tr} h^{-1}|^{2}}{({\rm tr} h+ {\rm tr} h^{-1})^{2}}-\frac{4}{\textmd{tr}h+\textmd{tr}h^{-1}} \cdot \textmd{tr}\{h(\sqrt{-1}\Lambda_{\omega }G_{H_{0}}-\lambda\cdot \textmd{Id}_{E}-\varepsilon \log h)\}\\
&+\frac{1}{\textmd{tr}h+\textmd{tr}h^{-1}} \cdot \textmd{tr} \{\sqrt{-1}\Lambda_{\omega }Dh\wedge h^{-1}D_{H_{0}}^{c}h+\sqrt{-1}\Lambda_{\omega }Dh^{-1}\wedge hD_{H}^{c}h^{-1} \}\\
&+\frac{4}{\textmd{tr}h+\textmd{tr}h^{-1}} \cdot \textmd{tr}\{h^{-1}(\sqrt{-1}\Lambda_{\omega }G_{H_{0}}-\lambda\cdot \textmd{Id}_{E}-\varepsilon \log h)\}\\
\geq& -\frac{4}{\textmd{tr}h+\textmd{tr}h^{-1}} \cdot \textmd{tr}\{h(\sqrt{-1}\Lambda_{\omega }G_{H_{0}}-\lambda\cdot \textmd{Id}_{E}-\varepsilon \log h)\}\\
&+\frac{4}{\textmd{tr}h+\textmd{tr}h^{-1}}\cdot \textmd{tr}\{h^{-1}(\sqrt{-1}\Lambda_{\omega }G_{H_{0}}-\lambda\cdot \textmd{Id}_{E}-\varepsilon \log h)\}\\
\geq&  -4|\sqrt{-1}\Lambda_{\omega } G_{H_{0}}-\lambda \cdot \textmd{Id}_E-\varepsilon \log h|_{H_{0}},
\end{split}
\end{equation*}
where the last inequality follows from $\frac{{\rm tr} h- {\rm tr} h^{-1}}{{\rm tr} h+ {\rm tr} h^{-1}}\leq 1$.
\end{proof}
Building on the analysis of the previous proposition, we establish the following result.
\begin{prop}  \label{P2}
Let $H$ and $H_{0}$ be two Hermitian metrics, then we have
$$\widetilde{\Delta} \log \{ {\rm tr} (H_0^{-1}H)+ {\rm tr} (H^{-1}H_0)\} \geq -4|\sqrt{-1}\Lambda_{\omega } G_{H_{0}}-\lambda \cdot {\rm Id}_E|_{H_{0}}-4|\sqrt{-1}\Lambda_{\omega } G_{H}-\lambda \cdot {\rm Id}_{E}|_{H}.$$
\end{prop}
\medskip
\begin{lem} [\protect Lemma 5.2 in {\cite{SIM}}] \label{SIMLEMMA}
Suppose that $(M, g)$ is a non-compact Gauduchon manifold equipped with an exhaustion function $\phi$ satisfying $\int_M|\widetilde{\Delta} \phi|\frac{\omega ^{n}}{n!}< \infty$. Let $\eta$ is a $(2n-1)$-form such that
$\int_M|\eta|^2\frac{\omega ^{n}}{n!}< \infty$. If $\mathrm{d}\eta$ is integrable, then
\begin{equation*}
\int_{M}\mathrm{d}\eta=0.
\end{equation*}
\end{lem}
\begin{prop} \label{idbundle01}
Suppose that $(E,D)$ is a projectively flat vector bundle with a fixed Hermitian metric $H_0$
over a Gauduchon manifold $(M, g)$. Let $H$ be a Hermitian metric on $E$, and define $s:=\log(H^{-1}_0H)$. Suppose one of the following two conditions holds:

(1)If $M$ is a compact manifold with smooth boundary $\partial M$,  and $H|_{\partial M}=H_0|_{\partial M}$.

(2)If $M$ is a non-compact manifold admitting an exhaustion function $\phi $ such that $\int_M|\widetilde{\Delta} \phi|\frac{\omega^{n}}{n!}<+\infty$. Additionally, assume that $|\mathrm{d}\omega^{n-1}|_{g}\in L^2(M)$, $s\in L^{\infty}(M)$ and $D_{H_0}^{c}s \in L^{2}(M)$.

Then, the following identity holds:
\begin{equation}\label{eq04021}
4 \int_M {\rm tr}(\Phi(H) s)\frac{\og^n}{n!}+\int_M\lag\Psi(s)(Ds), Ds\rag_{H_{0}}\frac{\og^n}{n!}+4 \varepsilon \|s\|^{2}_{L^2} = 4 \int_M {\rm tr}(\Phi(H_{0}) s)\frac{\og^n}{n!},
\end{equation}
where $\Phi(H)$ is defined as (\ref{wwwww2}) and $\Psi $ is the function defined in (\ref{eq3301}).
\end{prop}
\begin{proof}
Set $h=H^{-1}_0H =e^s$. By the definition, it holds that
\begin{equation}\label{def}
4\textmd{tr}\{(\Phi(H)-\Phi(H_0))s\}+4\varepsilon \|s\|^{2}_{L^2} =\langle \sqrt{-1}\Lambda_{\omega } D(h^{-1}D_{H_0}^{c}h),s\rangle_{H_0}.
\end{equation}
On the other hand, by using $\textmd{tr} (h^{-1}(D_{H_0}^{c}h)s)=\textmd{tr} (s D_{H_0}^{c}s)$ and $ \partial\bar{\partial}\omega^{n-1}=0$, we can obtain
\begin{equation*}
\begin{split}
&\int_{M}\langle \sqrt{-1}\Lambda_{\omega } (D(h^{-1}D_{H_0}^{c}h)),s\rangle_{H_0}\frac{\omega^n}{n!}\\
&=\int_M\ \sqrt{-1}d\ \textmd{tr}(sD_{H_0}^{c}s)\wedge\frac{\omega^{n-1}}{(n-1)!}
+\int_M\ \sqrt{-1} \textmd{tr}(h^{-1} D_{H_0}^{c}h \wedge Ds)\wedge\frac{\omega^{n-1}}{(n-1)!}\\
&=\int_M\ \sqrt{-1}d\ \{\textmd{tr}(sD_{H_0}^{c}s)\wedge\frac{\omega^{n-1}}{(n-1)!}\}
+\int_M\ \sqrt{-1} \textmd{tr}(sD_{H_0}^{c}s)\wedge d\ (\frac{\omega^{n-1}}{(n-1)!})\\
&~~~~+\int_M\ \sqrt{-1} \textmd{tr}(h^{-1} D_{H_0}^{c}h \wedge Ds)\wedge\frac{\omega^{n-1}}{(n-1)!}\\
&=\int_M\ \sqrt{-1}d\ \{\textmd{tr}(sD_{H_0}^{c}s)\wedge\frac{\omega^{n-1}}{(n-1)!}\}
+\int_M\ \frac{\sqrt{-1}}{2} d^{c}\ \textmd{tr}(s^{2})\wedge d\ (\frac{\omega^{n-1}}{(n-1)!})\\
&~~~~+\int_M\ \sqrt{-1} \textmd{tr}(h^{-1} D_{H_0}^{c}h \wedge Ds)\wedge\frac{\omega^{n-1}}{(n-1)!}\\
&=\int_M\ \sqrt{-1}d\ \{\textmd{tr}(sD_{H_0}^{c}s)\wedge\frac{\omega^{n-1}}{(n-1)!}\}
+\int_M\ \frac{\sqrt{-1}}{2} d^{c}\ \{ \textmd{tr}(s^{2})\wedge d\ (\frac{\omega^{n-1}}{(n-1)!}) \}\\
&+\int_M\ \sqrt{-1} \textmd{tr}(h^{-1} D_{H_0}^{c}h \wedge Ds)\wedge\frac{\omega^{n-1}}{(n-1)!}.
\end{split}
\end{equation*}
By applying the condition $s|_{\partial M}=0$ and Stokes formula in case (1), and utilizing Lemma \ref{SIMLEMMA} in  case (2), we derive
\begin{equation} \label{theta1}
\int_{M}\langle \sqrt{-1}\Lambda_{\omega } (D(h^{-1}D_{H_0}^{c}h)),s\rangle_{H_0}\frac{\omega^n}{n!}=\int_M\ \sqrt{-1} \textmd{tr}(h^{-1} D_{H_0}^{c}h \wedge Ds)\wedge\frac{\omega^{n-1}}{(n-1)!}.
\end{equation}
It was shown in \cite[Proposition 2.3]{cppcjzz} that
\begin{equation}\label{theta21}
\sqrt{-1} \Lambda_{\omega } \textmd{tr}(h^{-1} D_{H_0}^{c}h \wedge Ds)= - \lag\Psi(s)(Ds), Ds\rag_{H_{0}}.
\end{equation}

Therefore, by (\ref{def}), (\ref{theta1})and (\ref{theta21}) , we obtain
$$
4 \int_M \textmd{tr}(\Phi(H) s)\frac{\og^n}{n!}+\int_M\lag\Psi(s)(Ds), Ds\rag_{H_{0}}\frac{\og^n}{n!}+4 \varepsilon \|s\|^{2}_{L^2} = 4 \int_M \textmd{tr}(\Phi(H_{0}) s)\frac{\og^n}{n!}.
$$
\end{proof}

\section{The related heat flow and the perturbed equation on compact Gauduchon manifolds}

In this section, we investigate the existence of long-time solutions to the heat flow equation (\ref{Flow}) and the solvability of the perturbed equation (\ref{eq}) on compact Gauduchon manifolds.
Suppose that $M$ is a compact Gauduchon manifold with smooth boundary $\partial
M$, equipped with a fixed metric $\widetilde{H}$ on $\partial M$, and let $(E,D)$ be a projectively flat vector bundle over $M$.

Firstly, we study the following Dirichlet boundary value problem:
\begin{equation} \label{Flow2}
\begin{cases}
H^{-1}\frac{\partial H}{\partial t}=4(\sqrt{-1}\Lambda_{\omega } G_{H}-\lambda \cdot \textmd{Id}_E-\varepsilon \log (H_0^{-1}H)),\\
H(0)=H_0,\\
H|_{\partial M}=\widetilde{H}.
\end{cases}
\end{equation}
It is straightforward to know that the flow (\ref{Flow2}) is strictly parabolic equation. According to the theory of parabolic equations, the solution to the equation exists for a short time. Let $T$ be the maximal existence time of the solution.
\begin{prop} \label{shorttime}
For sufficiently small $T >0 $,
the equation \eqref{Flow2} admits a smooth solution defined for
$0\leq t <T$.
\end{prop}
As follows, we will give the long-time existence of the heat flow.
\begin{lem}  \label{C0ofdistance1}
Suppose that a smooth solution $H(t)$ to the equation
\eqref{Flow2} is defined for $0\leq t < T <+\infty $. Then
$H(t)$ converges to a continuous
non-degenerate metric $H_{T}$ in $C^{0}$-topology as $t\rightarrow T$.
\end{lem}

\begin{proof}
By the continuity at $t=0$, for any $\varepsilon >0$, there exists $\delta>0$ such that

$$
\sup_{M} \sigma (H(t_{0}) , H(t'_{0})) < \varepsilon
$$
for $0< t_{0}, t'_{0} <\delta < T $. By Proposition \ref{P243} and the maximum
principle, it follows that
$$
\sup_{M} \sigma (H(t) , H(t'))<\varepsilon
$$
for all $ T- \delta < t, t' <T $. This shows that $H(t)$ forms a
uniformly Cauchy sequence and converges to a continuous limiting
metric $H_{T}$. Furthermore, by Proposition \ref{P1}, we conclude
that
\begin{equation*}\label{z1}
\sup_{M\times [0, T)}|\sqrt{-1}\Lambda_{\omega } G_{H(t)}-\lambda \cdot \textmd{Id}_E-\varepsilon \log (H_0^{-1}H(t))|_{H(t)}
\end{equation*}
is bounded uniformly, i.e.,
\begin{equation*}\label{z1}
\sup_{M\times [0, T)}|\sqrt{-1}\Lambda_{\omega } G_{H(t)}-\lambda \cdot \textmd{Id}_E-\varepsilon \log (H_0^{-1}H(t))|_{H(t)}<A,
\end{equation*}
where $A$ is a uniform constant depending only on the initial metric $H_{0}$. By using the inequalities
\begin{equation*}
\left|\frac{\partial }{\partial t}(\log \textmd{tr} h)\right| \leq 4|\sqrt{-1}\Lambda_{\omega } G_{H(t)}-\lambda \cdot \textmd{Id}_E-\varepsilon \log (H_0^{-1}H)|_{H(t)}
\end{equation*}
and
\begin{equation*}
\left|\frac{\partial }{\partial t}(\log \textmd{tr} h^{-1})\right| \leq
4|\sqrt{-1}\Lambda_{\omega } G_{H(t)}-\lambda \cdot \textmd{Id}_E-\varepsilon \log (H_0^{-1}H)|_{H(t)},
\end{equation*}
we can derive
$$\sigma (H(t), H_{0})\leq 2\mathrm{rank}(E)(e^{4At}-1).$$
This implies that $\sigma (H(t), H_{0})$ is uniformly bounded on $M\times [0, T)$ with respect to $t$. Consequently, the limiting metric $H_{T}$ is
non-degenerate.
\end{proof}

For further discussion, we can establish the following lemma via using the same method as in \cite{Z}.
\begin{lem} \label{www1}
Suppose that $(M,g)$ is a compact Hermitian manifold with non-empty boundary. Let $H(t)$, $0\leq t
<T$, be a family of Hermitian metrics on projectively flat vector bundle $E$ over $M$,
and assume that $H(t)$ converges to a continuous
metric $H_{T}$ in the $C^{0}$-topology as $t\rightarrow T$. Suppose $\sup\limits_M |\Lambda_{\omega }
G_{H(t)}|_{H_{0}}$ is uniformly bounded in $t$, then

(1)if $H(t)$ satisfies Dirichlet boundary condition, then $H(t)$ is
 bounded in $L_{2}^{p}$ for any
$1<p<+\infty $ uniformly in $t$.

(2) if  $H(t)$ does not satisfy Dirichlet boundary condition, then $H(t)$ is
bounded in $L_{2,loc}^{p}$ in the interior for any
$1<p<+\infty $ uniformly in $t$.\\
\end{lem}

Combining the above results with the regularity theory of parabolic equations, we can obtain the long-time existence of the solution.
\begin{prop} \label{compactlongtime}
For $\varepsilon \geq 0$, the equation \eqref{Flow2}
admits a unique long-time solution $H(t)$ that exist for $0\leq t <+\infty$.
\end{prop}
\begin{proof}
By Proposition \ref{shorttime}, the solution $H(t)$ to the equation \eqref{Flow2} exists for a short time defined for $0\leq
t<T < +\infty $. From Lemma \ref{C0ofdistance1}, $H(t)$ converges in $C^{0}$-topology to a continuous
non-degenerate metric $H_{T}$ as $t\rightarrow T
$. The inequality \eqref{dec} and the maximum principle imply that $\sup\limits_M |\Lambda_{\omega }
G_{H(t)}|_{H_{0}}$ is uniformly bounded on $[0, T)$. Therefore, by Lemma \ref{www1}, $H(t)$ is
bounded in $L_{2}^{p}$ for any
$1<p<+\infty $  uniformly in $t$. Since \eqref{Flow2} is quadratic in the first derivative of $H$, we can
use Hamilton's method \cite{HAMILTON} to imply that $H(t)\rightarrow H_{T}$
in $C^{\infty}$-topology, and the solution can be extended beyond $T$. Thus,
\eqref{Flow2} admits a solution $H(t)$
for all time.

The uniqueness of the solution follows directly from Proposition \ref{P243} and the maximum principle.
\end{proof}

In the following, we also obtain the $C^{0}$-estimate of $H(t)$ using an alternative approach.
\begin{prop} \label{noncompactc0}
Suppose that $H(t)$ is a long-time solution to the heat flow (\ref{Flow2}) on $M$, then we have
\begin{equation}\label{c0key}|\log h|_{H_0}\leq \frac{1}{\varepsilon}\max_{M}|\Phi(H_0)|_{H_0},\end{equation}
where $h(t)=H_{0}^{-1}H(t)$.
\end{prop}

\begin{proof}
By direct calculation, we have
\begin{equation} \label{JSJJ11}
\begin{split}
\langle H^{-1}\frac{\partial H}{\partial t}, \log h \rangle_{H_0} &= 4\langle \sqrt{-1}\Lambda_{\omega} G_{H} - \lambda \cdot \mathrm{Id}_E - \varepsilon \log h, \log h \rangle_{H_0} \\
&= \langle 4\Phi(H_0) + \sqrt{-1}\Lambda_{\omega}(D(h^{-1}D_{H_0}^{c}h) - 4\varepsilon \log h, \log h \rangle_{H_0}.
\end{split}
\end{equation}
We can easily check that
\begin{equation}\label{JSJJ22}
\langle H^{-1}\frac{\partial H}{\partial t},\log h\rangle_{H_0}=\langle h^{-1}\frac{\partial h}{\partial t},\log h\rangle_{H_0}=\frac{1}{2}\frac{\partial}{\partial t}|\log h|^2_{H_0}
\end{equation}
and
\begin{equation}\label{JSJJ33}
\langle \sqrt{-1}\Lambda_{\omega }D(h^{-1}D_{H_0}^{c}h),\log h\rangle_{H_0}\leq \frac{1}{2}\widetilde{\Delta}(|\log h|^2_{H_0}).
\end{equation}
From \eqref{JSJJ11}, \eqref{JSJJ22} and \eqref{JSJJ33}, it follows that
\begin{equation*}
\frac{1}{8}(\frac{\partial}{\partial t}-\widetilde{\Delta})(|\log h|^2_{H_0})\leq |\Phi(H_0)|_{H_0}|\log h|_{H_0}-\varepsilon |\log h|^2_{H_0},
\end{equation*}
Finally, by using the maximum principle, we obtain \eqref{c0key}.
\end{proof}

\begin{prop} \label{noncompactc1} Suppose that $H(t)$ is a long-time solution to the heat flow (\ref{Flow2}) on compact Gauduchon manifold $M$ with smooth boundary $\partial M$. Set $h(t)=H_{0}^{-1}H(t)$. Assume there exists a constant $\overline{C}_0$ with
$$\sup_{(x,t)\in M \times[0,+\infty )}|\log h(t)|_{H_0}\leq \overline{C}_0.
$$
Then, for any compact subset $\Omega \subset M $, there exists a uniform constant $\overline{C}_1$, depending only on  $d^{-1}$, $\overline{C}_0$ and the geometry of $\tilde{\Omega }$, such that
\begin{equation}\label{CC1}\sup_{(x,t)\in \Omega \times[0,+\infty )}|\psi_{H}^{1,0}|_{H_0,\omega}\leq \overline{C}_1,
\end{equation}
where $d$ is the distance from $\Omega $ to $\partial M$ and $\tilde{\Omega }=\{x \in M|\mathrm{dist}(x, \Omega)\leq \frac{1}{2}d\}$.
\end{prop}

\begin{proof}
A straightforward calculation shows that
\begin{equation} \label{noncompactc1eq1}
\begin{split}
(\widetilde{\Delta}-\frac{\partial}{\partial t}) \textmd{tr}h &= \textmd{tr}(\sqrt{-1}\Lambda_{\omega }(Dh\wedge h^{-1}D_{H_0}^{c}h))-4 \textmd{tr}(h\Phi(H_{0}))
+4\varepsilon \textmd{tr}(h\log h)\\
&=|Dh\cdot h^{-\frac{1}{2}}|^{2}_{H_{0},\omega}-4 \textmd{tr}(h\Phi(H_{0}))
+4\varepsilon \textmd{tr}(h\log h),
\end{split}
\end{equation}

\begin{equation} \label{noncomp33}
\frac{\partial \psi_{H}^{1,0}}{\partial t}=-\frac{1}{2}\partial_{H}(h^{-1}(t)\frac{\partial h(t)}{\partial t})+\frac{1}{2}[\psi_{H}^{1,0},h^{-1}(t)\frac{\partial h(t)}{\partial t}],
\end{equation}
and
\begin{equation} \label{noncomp35}
\begin{split}
\widetilde{\Delta} |\psi_{H}^{1,0}|_{H,\omega}^{2}\geq &2|\nabla \psi_{H}^{1,0}|_{H,\omega}^{2}+2|[\psi_{H}^{1,0},\psi_{H}^{0,1}]|_{H,\omega}^{2}+4|\psi_{H}^{1,0}\wedge\psi_{H}^{1,0}|_{H,\omega}^{2}\\ &-4\textmd{Re} \langle\partial_{H}(\sqrt{-1}\Lambda_{\omega}G_{H}),\psi_{H}^{1,0}\rangle_{H,\omega}-2\check{C}_1|\nabla\psi_{H}^{1,0}|_{H,\omega}|\psi_{H}^{1,0}|_{H,\omega}
-2|R|\cdot|\psi_{H}^{1,0}|_{H,\omega}^{2},
\end{split}
\end{equation}
where $\check{C}_1$ is a positive constant depending only on the geometry of $\tilde{\Omega }$, and $R$ denotes the curvature of the Chern connection.
By \eqref{noncomp33} and \eqref{noncomp35}, we can obtain
\begin{equation} \label{noncompactc1eq3}
(\widetilde{\Delta}-\frac{\partial}{\partial t})|\psi_{H}^{1,0}|^2_{H,\omega}\geq |\nabla \psi_{H}^{1,0}|_{H,\omega}^{2}-\check{C}_2|\psi_{H}^{1,0}|_{H,\omega}^{2}-4\varepsilon |\partial_{H}(\log(h))|_{H,\omega}|\psi_{H}^{1,0}|_{H,\omega}
\end{equation}
on $\tilde{\Omega }\times [0,+\infty )$, where $\check{C}_2$ is a uniform constant depending only on $\overline{C}_0$ and the geometry of $\tilde{\Omega }$.
Notice that
	\begin{equation}
	|\partial_{H}(\log(h))|_{H,\omega}\leq \check{C}_3|h^{-1}\partial_{H}h|_{H,\omega}\leq \check{C}_4|\psi_{H}^{1,0}|_{H,\omega}+\check{C}_5,
	\end{equation}
	then
	\begin{equation}
	\begin{split}
	(\widetilde{\Delta}-\frac{\partial}{\partial t})|\psi_{H}^{1,0}|^2_{H,\omega}&\geq |\nabla \psi_{H}^{1,0}|_{H,\omega}^{2}-\check{C}_6|\psi_{H}^{1,0}|_{H,\omega}^{2}-\check{C}_7|\psi_{H}^{1,0}|_{H,\omega}\\
&\geq |\nabla \psi_{H}^{1,0}|_{H,\omega}^{2}-C_{1}|\psi_{H}^{1,0}|_{H,\omega}^{2}-C_{2}.
	\end{split}
	\end{equation}
where $\check{C}_3,\check{C}_4,\check{C}_5,\check{C}_6 $ and $\check{C}_7$ are the constants depending only on $\overline{C}_0$ and $\tilde{\Omega }$.

Define $\overline{\Omega }=\{x \in M | \mathrm{dist} (x, \Omega )\leq \frac{1}{4}d\}$. Let $0\leq \varphi_1\leq \varphi_2\leq 1$ be non-negative cut-off functions satisfying
 \begin{gather*}
     \varphi_1
     =\begin{cases}
       0, & x\in M\backslash \overline{\Omega }, \\
       1, & x\in \Omega,
     \end{cases}
  \end{gather*}
   \begin{gather*}
    \varphi_2
     =\begin{cases}
       0, & x\in M\backslash \tilde{\Omega }, \\
       1, & x\in \overline{\Omega },
     \end{cases}
  \end{gather*}

  $$|\textmd{d} \varphi_i|^2+|\widetilde{\Delta} \varphi_i|\leq c, \ \ i=1,2.$$
Consider the test function defined as
$$f(\cdot,t)=\varphi_1^2|\psi_{H}^{1,0}|^2_{H,\omega}+W\varphi_2^2\textmd{tr}h,$$
where $W$ is a constant to be determined later. Combining \eqref{noncompactc1eq1} and \eqref{noncompactc1eq3}, we obtain
\begin{equation}
\begin{split}
  (\tilde{\Delta}-\frac{\partial}{\partial t})f\geq & \varphi_1^2(|\nabla \psi_{H}^{1,0}|_{H,\omega}^{2}-C_{1}|\psi_{H}^{1,0}|_{H,\omega}^{2}-C_{2})+\tilde{\Delta}
  \varphi_1^2|\psi_{H}^{1,0}|_{H,\omega}^{2}\\
  &+4\langle\varphi_1\nabla \varphi_1,\nabla|\psi_{H}^{1,0}|_{H,\omega}^{2}\rangle+W\tilde{\Delta}\varphi_2^2{\rm tr}h+4W\langle\varphi_2\nabla\varphi_2,\nabla{\rm tr}h\rangle\\
  &+W\varphi_2^2 \{|Dh\cdot h^{-\frac{1}{2}}|^{2}_{H_{0},\omega}-4 \textmd{tr}(h\Phi(H_{0}))
+4\varepsilon \textmd{tr}(h\log h)\}.
\end{split}
\end{equation}

We make use of the following inequalities
\begin{equation}
\begin{split}
  4\langle\varphi_1\nabla\varphi_1,\nabla|\psi_{H}^{1,0}|_{H,\omega}^{2}\rangle&\geq -8\varphi_1|\nabla \varphi_1||\psi_{H}^{1,0}|_{H,\omega}|\nabla_H\psi_{H}^{1,0}|_{H,\omega}\\
  &\geq-\varphi_1^2|\nabla_H\psi_{H}^{1,0}|_{H,\omega}^2-16|\nabla\varphi_1|^2|\psi_{H}^{1,0}|_{H,\omega}^{2},
\end{split}
\end{equation}
\begin{equation}
  4W\langle\varphi_2\nabla\varphi_2,\nabla{\rm tr}h\rangle \geq-\varphi_2^2|\nabla{\rm tr}h|^2-4W^2|\nabla\varphi_2|^2,
\end{equation}
\begin{equation}
\begin{split}
  |Dh\cdot h^{-\frac{1}{2}}|^{2}_{H_{0},\omega}\geq \check{C}_8 |Dh\cdot h^{-1}|^{2}_{H,\omega}\geq \check{C}_9 |\psi_{H}^{1,0}|_{H,\omega}^{2}-\check{C}_{10}.
\end{split}
\end{equation}
where $\check{C}_8,\check{C}_9 $ and $\check{C}_{10}$ are the constants depending only on $\overline{C}_0$ and $\tilde{\Omega }$.

Hence, we can obtain
\begin{equation}
\begin{split}
  (\tilde{\Delta}-\frac{\partial}{\partial t})f\geq & \varphi_1^2(-C_{1}|\psi_{H}^{1,0}|_{H,\omega}^{2}-C_{2})+\tilde{\Delta}
  \varphi_1^2|\psi_{H}^{1,0}|_{H,\omega}^{2}\\
  &-16|\nabla\varphi_1|^2|\psi_{H}^{1,0}|_{H,\omega}^{2}+W\tilde{\Delta}\varphi_2^2{\rm tr}h
  -\varphi_2^2|\nabla{\rm tr}h|^2-4W^2|\nabla\varphi_2|^2\\
  &+W\varphi_2^2 \{\check{C}_9 |\psi_{H}^{1,0}|_{H,\omega}^{2}-\check{C}_{10}-4 \textmd{tr}(h\Phi(H_{0}))
+4\varepsilon \textmd{tr}(h\log h)\}\\
 &\geq \varphi_2^2(-C_{1}|\psi_{H}^{1,0}|_{H,\omega}^{2}-C_{2})-C_{3}\varphi_2^2 |\psi_{H}^{1,0}|_{H,\omega}^{2}-16C_{4} \varphi_2^2|\psi_{H}^{1,0}|_{H,\omega}^{2}\\
 &-C_{5}C_{6}W\varphi_2^2-C_{7}\varphi_2^2-4C_{8}W^2\varphi_2^2+\varphi_2^2 \check{C}_9 W|\psi_{H}^{1,0}|_{H,\omega}^{2}-C_{9}W\varphi_2^2\\
 &\geq  \varphi_2^2 |\psi_{H}^{1,0}|_{H,\omega}^{2}(-C_{1}-C_{3}-16C_{4}+\check{C}_9 W)-C_{2}-C_{5}C_{6}W-C_{7}\\
 &-4C_{8}W^2-C_{9}W.
\end{split}
\end{equation}
By choosing $W$ such that $-C_{1}-C_{3}-16C_{4}+\check{C}_9 W=1$, we have
\begin{equation} \label{noncompactc1eq4}
(\widetilde{\Delta}-\frac{\partial}{\partial t})f\geq \psi^2_2|\psi_{H}^{1,0}|_{H,\omega}^{2}-\widetilde{C}_0 \quad  \quad  \text{on } M\times [0,+\infty ).
\end{equation}
 Let $f(y,t_0)=\max\limits_{M\times [0,+\infty)}f$. From the definition of $\psi_i$ and the uniform $C^0$-bound of $h(t)$, we can assume
$$(y,t_0)\in \overline{\Omega}\times (0,+\infty).$$
Then, \eqref{noncompactc1eq4} implies
\begin{equation*}
|\psi_{H}^{1,0}(t_0)|^2_{H(t_0),\omega}(y)\leq \widetilde{C}_0.
\end{equation*}
Therefore, there exists a uniform constant $\overline{C}_1$ depending only on $d^{-1}$, $\overline{C}_0$ and the geometry of $\tilde{\Omega }$ such that
\begin{equation}\label{CC1}\sup_{(x,t)\in \Omega \times[0,+\infty )}|\psi_{H}^{1,0}|_{H_0,\omega}\leq \overline{C}_1.
\end{equation}
\end{proof}

Secondly, we will solve the perturbed equation (\ref{eq}) over the compact Gauduchon manifold and establish the following theorem.

\begin{thm} \label{comthm11}
Suppose that $(E,D)$ be a projectively flat vector bundle equipped with a fixed Hermitian metric $H_0$
over the compact Gauduchon manifold $M$ with non-empty
boundary $\partial M$. Then there exists a unique
Hermitian metric $H$ on $E$ satisfying
\begin{equation} \label{dc11}
\begin{cases}
\sqrt{-1}\Lambda_{\omega } G_{H}-\lambda \cdot \mathrm{Id}_E-\varepsilon \log (H_0^{-1}H)=0,\\
H|_{\partial M}=H_{0},
\end{cases}
\end{equation}
for any $\varepsilon \geq 0$.
When $\varepsilon >0$,  the following estimates hold:
\begin{equation}\label{DC01}
\max_{x\in M}|s|_{H_0}(x)\leq \frac{1}{\varepsilon}\max_{M}|\Phi(H_0)|_{H_0}.
\end{equation}
and
\begin{equation} \label{eq0431}
\|Ds\|_{L^2(M)}=\|D_{H_0}^{c}s\|_{L^2(M)}\leq C(\varepsilon^{-1},\Phi(H_0),\mathrm{Vol}(M)),
\end{equation}
where $s=\log (H_{0}^{-1}H)$. Moreover, if the Hermitian metric $H_{0}$ satisfies
\begin{equation}\label{trace31}
\mathrm{tr}(\sqrt{-1}\Lambda_{\omega } G_{H}-\lambda \cdot \mathrm{Id}_E)=0,
\end{equation}
then $\mathrm{tr}(s)=0$ and $H$ also satisfies (\ref{trace31}).
\end{thm}

\begin{proof}
Assume the initial Hermitian metric $H_{0}$ satisfies (\ref{trace31}), then we obtain
\begin{equation}
\{e^{4\varepsilon t}\cdot \mathrm{tr} (\sqrt{-1}\Lambda_{\omega }G_{H(t)}-\lambda \cdot \mathrm{Id}_E-\varepsilon \log h(t))\}=0
\end{equation}
on $M\times\{t=0\}$ and $\partial M\times[0,t)$. Combining the evolution equation (\ref{trace}) and the maximum principle, we have
\[
\mathrm{tr}\left\{\sqrt{-1}\Lambda_{\omega}G_{H(t)} - \lambda \cdot \mathrm{Id}_E - \varepsilon\log(H_{0}^{-1}H(t))\right\} = 0 \quad \text{on } M\times[0,t).
\]
Hence, it follows that
\begin{equation*}
\tr (\log H_{0}^{-1} H(t))=0, \  \ \tr (\sqrt{-1}\Lambda_{\omega } G_{H(t)}-\lambda \cdot \textmd{Id}_E)=0.
\end{equation*}

Following the method in \cite[Proposition 1.8, Chapter 5]{T}, we can solve the following Dirichlet boundary value problem on the manifold $M$:
\begin{equation} \label{thmeq2}
\widetilde{\triangle }v =-
|\sqrt{-1}\Lambda_{\omega } G_{H_0}-\lambda \cdot \textmd{Id}_E|_{H_{0}}, \ \
v|_{\partial M}=0.
\end{equation}
Define the function
$$w(x, t)=\int_{0}^{t}|\sqrt{-1}\Lambda_{\omega } G_{H}-\lambda \cdot \textmd{Id}_E-\varepsilon \log (H_0^{-1}H)|_{H}(x, \rho) \mathrm{d}\rho -v(x).$$
From (\ref{dec}), (\ref{thmeq2}), and the boundary
condition of $H$, we conclude that, for $t>0$,
\[
\bigl|\sqrt{-1}\Lambda_{\omega} G_{H} - \lambda \mathrm{Id}_E - \varepsilon \log (H_0^{-1}H)\bigr|_{H} (x,t) = 0 \quad \text{on } \partial M.
\]
Thus, $w(x, t)$
satisfies
\begin{equation}\label{xqbh222}
(\widetilde{\triangle } -\frac{\partial
}{\partial t })w(x, t)\geq 0, \ \ w(x, 0)=-v(x), \ \ w(x, t)|_{\partial M }=0.
\end{equation}
Applying the maximum principle to (\ref{xqbh222}), we obtain
\begin{equation} \label{thmeq3}
\int_{0}^{t}|\sqrt{-1}\Lambda_{\omega } G_{H}-\lambda \cdot \textmd{Id}_E-\varepsilon \log (H_0^{-1}H)|_{H}(x, \rho) \mathrm{d}\rho \leq
\sup_{y\in M} v(y)
\end{equation}
for any $x\in M$, $0<t <+\infty $.

Fix $t_{1}\leq t \leq t_{2}$ and define $\tilde{h}(x, t)= H^{-1}(x,
t_{1})H(x, t)$. A direct calculation shows
\begin{equation}\label{xqbhd562}
\frac{\partial }{\partial t} \log \textmd{tr}(\tilde{h})\leq
4|\sqrt{-1}\Lambda_{\omega } G_{H}-\lambda \cdot \textmd{Id}_E-\varepsilon \log (H_0^{-1}H)|_{H}.
\end{equation}
Integrating (\ref{xqbhd562}) from $t_{1}$ to $t$, we derive
\begin{equation}\label{hlhl12}
\begin{split}
&\textmd{tr} (H^{-1}(x, t_{1})H(x, t))\\
&~~\leq r\exp{(4\int_{t_{1}}^{t}|\sqrt{-1}\Lambda_{\omega } (G_{H}-\lambda \cdot \textmd{Id}_E-\varepsilon \log (H_0^{-1}H)|_{H}
\mathrm{d}\rho)}.
\end{split}
\end{equation}
Similarly, by reversing the roles of $t_{1}$ and $t$, we have
\begin{equation} \label{hlhl25}
\begin{split}
&\textmd{tr} (H^{-1}(x, t)H(x, t_{1}))\\
&~~\leq r\exp{(4\int_{t_{1}}^{t}|\sqrt{-1}\Lambda_{\omega } G_{H}-\lambda \cdot \textmd{Id}_E-\varepsilon \log (H_0^{-1}H)|_{H}
\mathrm{d}\rho)}.
\end{split}
\end{equation}
Combining (\ref{hlhl12}) and (\ref{hlhl25}), we conclude
\begin{equation} \label{thmeq4}
\begin{split}
&\sigma (H(x , t), H(x , t_{1}) )\\
&~~\leq 2r
\{\exp{(4\int_{t_{1}}^{t}|\sqrt{-1}\Lambda_{\omega } G_{H}-\lambda \cdot \textmd{Id}_E-\varepsilon \log (H_0^{-1}H)|_{H}
\mathrm{d}\rho)} -1 \}.
\end{split}
\end{equation}
From the inequalities (\ref{thmeq3}) and (\ref{thmeq4}), we conclude \( H(t) \) converges to a continuous metric \( H_{\infty} \) in the \( C^0 \)-topology as \( t \to +\infty \). By Lemma \ref{www1}, for any \( 1 < p < +\infty \), \( H(t) \) is uniformly bounded in \( C^1_{\text{loc}} \) and also uniformly bounded in \( L_{2, \text{loc}}^p \). Additionally, \( |H^{-1} \frac{\partial H}{\partial t}| \) is uniformly bounded. By standard elliptic regularity theory, there exists a subsequence of \( H(t) \) that converges in the \( C^{\infty}_{\text{loc}} \)-topology to the metric \(H_{\infty} \). Moreover, from (\ref{thmeq3}), we conclude that
\begin{equation} \label{thmeqww3}
\sqrt{-1} \Lambda_\omega G_{H_{\infty}} - \lambda \cdot \mathrm{Id}_E - \varepsilon \log(H_0^{-1}H_{\infty}) = 0,
\end{equation}
i.e., \( H_{\infty} \) is the desired Hermitian metric satisfying the boundary condition. The uniqueness of the solution follows from Corollary \ref{uniq} and the maximum principle.

When $\varepsilon >0$, the inequality (\ref {c0key}) in Proposition \ref{noncompactc0} directly implies (\ref{DC01}). According to the definition,  we can obtain
\begin{equation*}
|Ds|_{H_{0},\omega}^{2}\leq \tilde{C}\langle \Psi(s)(Ds),Ds\rangle_{H_0,\omega},
\end{equation*}
where $\tilde{C}$ denotes a positive constant that depends only on the $L^{\infty}$-bound of $s$. Integrating both sides over $M$ and applying (\ref{eq04021}), (\ref{c0key}),(\ref{DC01}), we have
\begin{equation}\label{eq0411}
\begin{split}
\int_{M} |Ds|_{H_{0},\omega}^{2}\frac{\omega^n}{n!}&\leq \tilde{C}\int_M
\langle \Psi(s)(Ds),Ds\rangle_{H_0,\omega}\frac{\omega^n}{n!}\\
&=\tilde{C}\int_M 4\ (\mathrm{tr}(\Phi(H_0)s)-\varepsilon |s|_{H_{0}}^{2})\frac{\omega^n}{n!}\\
& \leq \tilde{C}\frac{4}{\varepsilon }\cdot \sup_{M}|\Phi(H_0)|_{H_0}^{2}\cdot \mathrm{Vol}(M, g).\\
\end{split}
\end{equation}
This completes the proof of (\ref{eq0431}).
\end{proof}

\section{The related heat flow and the perturbed equation on the non-compact Gauduchon manifold}

In this section, we study the existence of long-time solutions to the heat flow equation (\ref{Flow2}) and the solvability of the perturbed equation (\ref{eq}) on a class of non-compact Gauduchon manifold $(M, g)$. Assume that $(M, g)$ satisfies the Assumption 1 and the Assumption 2.
Let $H_{0}$ be an initial Hermitian metric on the vector bundle $E$.

We fix a real number $\varphi$ and define the compact subset $M_{\varphi}=\{x\in M |\phi(x)\leq \varphi \}$, where $\phi$ is the exhaustion function on $M$. Consider the following Dirichlet boundary condition
\begin{equation}\label{D1}
H|_{\partial M_{\varphi}}=H_0|_{\partial M_{\varphi}}.
\end{equation}
According to Proposition \ref{compactlongtime}, on each $M_{\varphi}$, there exists a unique long-time solution $H_{\varphi}(t)$  to the following Dirichlet problem
\begin{equation} \label{Flow282}
\begin{cases}
H_{\varphi}^{-1}\frac{\partial H_{\varphi}}{\partial t}=4(\sqrt{-1}\Lambda_{\omega } G_{H_{\varphi}}-\lambda \cdot \textmd{Id}_E-\varepsilon \log (H_0^{-1}H_{\varphi})),\\
H_{\varphi}(0)=H_0,\\
H_{\varphi}|_{\partial M_{\varphi}}=H_0|_{\partial M_{\varphi}},
\end{cases}
\end{equation}
for $0\leq t<+\infty$.

\begin{lem} [\protect Lemma 6.7 in {\cite{SIM}}] \label{SIMLEM}
Let $u$ be a function on $M_{\varphi}\times[0,T]$ satisfying
\begin{equation} \label{FC111}
\begin{cases}
(\widetilde{\Delta}-\frac{\partial}{\partial t})u\geq 0,\\ \ \ \ u|_{t=0}=0,
\end{cases}
\end{equation}
and assume there is a bound $\sup\limits_{M_{\varphi}}u\leq C_1$. Then we have
$$u(x,t)\leq \frac{C_1}{\varphi}(\phi(x)+C_2t),$$
where $C_2$ is the bound of $\widetilde{\Delta}\phi$ in Assumption 2.

\end{lem}

\medskip

In what follows, a constant $C$ is assumed to exist such that
 $$\sup_{M}|\Phi(H_0)|_{H_0}\leq C.$$

$\bullet$ $\textbf{Step 1} \ $ \ \, We will obtain the $C^{0}$-convergence of $H_{\varphi}$ over any compact subset $\Omega\subset M$ within a finite time interval. For every $\Omega \subset M$, there exists a constant $\varphi_{0}$ such that $\Omega \subset M_{\varphi_{0}}$.
Suppose $H_{\varphi_{1}}(t)$ and $H_{\varphi_2}(t)$ be the long-time solutions for the heat flow equation (\ref{Flow282}) with the Dirichlet boundary condition (\ref{D1}), where $\varphi_{0}<\varphi_{1}<\varphi_{2}$. Define the function $\nu=\sigma(H_{\varphi_{1}},H_{\varphi_2})$. Proposition \ref{noncompactc0} provides a uniform bound on $\nu$, and $\nu$ satisfies (\ref{FC111}). Applying Lemma \ref{SIMLEM}, we obtain
\begin{equation}
    \sigma(H_{\varphi}, H_{\varphi_1}) \leq C_{1} \frac{(\varphi_{0} + C_{2} T)}{\varphi}
    \quad \text{on } M_{\varphi_{0}} \times [0, T].
\end{equation}
Consequently, $H_{\varphi}$ is a Cauchy sequence on $M_{\varphi_{0}}\times [0, T]$ as $\varphi\rightarrow \infty$.

$\bullet$ $\textbf{Step 2}$ \ \, We will establish the $C^{\infty}$-convergence of $H_{\varphi}$ on any compact set $\Omega$. Proposition \ref{noncompactc0} and Proposition \ref{noncompactc1} provide uniform $C^0$-bounds and local $C^1$-estimates for $H_{\varphi}(t)$. Using the standard Schauder estimate of the parabolic equation, we can derive the local uniform $C^{\infty}$-estimate of $H_{\varphi }(t)$. However, it is important to note that utilizing the parabolic Schauder estimate, we only obtain the uniform $C^{\infty}$-estimate of $h(t)$ on $M_{\varphi}\times[\tau,T]$, where $\tau>0$, and the estimate depends on $\tau^{-1}$. To address this, we employ the maximum principle to obtain a local uniform bound on the curvature $|G_H|_H$, followed by elliptic estimates to derive local uniform $C^{\infty}$-estimates. This step is omitted here, as it closely follows the arguments in \protect{\cite[Lemma 2.5]{LZZ}}.
By selecting a subsequence $\varphi \rightarrow +\infty $, we conclude that $H_{\varphi}(t)$  converge in $C_{loc}^{\infty}$-topology to a long-time solution $H(t)$ for the heat flow equation (\ref{Flow2}) on $M$. Thus, we derive the following theorem.

\begin{thm} \label{noncompactthm33}
Let $(M,g)$ be a non-compact Gauduchon manifold satisfying the Assumption 1 and the Assumption 2, $(E,D)$ be a projectively flat vector bundle equipped with an initial Hermitian metric $H_0$ over $M$. Assume that $H_0$ satisfies $$\sup\limits_{M}|\sqrt{-1}\Lambda_{\omega } G_{H_{0}}|_{H_0}<+\infty.$$ Then, the following hold:

(1) for $\varepsilon \geq 0$, the flow \eqref{Flow2} admits a long-time solution $H(t)$  on the entire manifold $M$ with the initial Hermitian metric $H_0$.

(2) for $\varepsilon > 0$,
\begin{equation}\label{c0key2}\sup_{(x,t)\in M\times[0,+\infty )}|\log h|_{H_0}(x,t)\leq \frac{1}{\varepsilon}\sup_{M}|\Phi(H_0)|_{H_0},\end{equation}
and
\begin{equation} \label{eq122431}
\|Ds\|_{L^2(M)}=\|D_{H_0}^{c}s\|_{L^2(M)}\leq C(\varepsilon^{-1},\Phi(H_0),\mathrm{Vol}(M)),
\end{equation}
where $s=\log (H_{0}^{-1}H)$.
\end{thm}

Here, we will investigate the solvability of the perturbed equation (\ref{eq}) on a non-compact Gauduchon manifold $M$. Let $\{M_{\varphi}\}^{\infty}_{\varphi=1}$ denote an exhaustion sequence of compact sub-domains of $M$. Assume that $(E,D)$ be a projectively flat vector bundle with a fixed Hermitian metric $H_0$. According to Theorem \ref{comthm11}, we know that there exists a Hermitian metric $H_{\varphi}(x)$ such that
\begin{equation*}\begin{cases}
\sqrt{-1}\Lambda_{\omega } G_{H_{\varphi}}-\lambda \cdot \textmd{Id}_E-\varepsilon \log (H_0^{-1}H_{\varphi})=0, \forall x\in M_{\varphi},\\
H_{\varphi}(x)|_{\partial M_{\varphi}}=H_0(x).
\end{cases}
\end{equation*}
To extend $H_{\varphi}(x)$ to the entire manifold $M$, we require certain priori estimates. The key step lies in establishing the $C^0$-estimate.
Define $h_{\varphi}=H_0^{-1}H_{\varphi}$. Theorem \ref{comthm11} yields the following estimate:
\begin{equation*}
\sup_{x\in M_{\varphi}}|\log h_{\varphi}|_{H_0}(x)\leq \frac{1}{\varepsilon}\max_{M_{\varphi}}|\Phi(H_0)|_{H_0}.
\end{equation*}
For any compact subset $\Omega\subset M$, there exists a $\varphi_{0}$ such that $\Omega\subset M_{\varphi_{0}}\subset M_{\varphi}$.
By Proposition \ref{noncompactc1}, we obtain the following estimate:
\begin{equation*}\label{CC10}
\sup_{x\in \Omega}|\psi_{H_{\varphi}}^{1,0}|_{H_0,\omega}\leq \hat{C}_1, \quad \text{for all } \varphi >\varphi_{0},
\end{equation*}
where $\hat{C}_{1}$ denotes a uniform constant independent of the choice of $\varphi$. By combining the perturbed equation (\ref{eq}) with standard elliptic theory, we obtain uniform local higher order estimates. By taking a subsequence, $H_{\varphi}$ converge to a limit metric $H_{\infty}$ in the $C_{loc}^{\infty}$-topology, which solves the equation (\ref{eq}) on the entire manifold $M$. So we have following theorem.

\begin{thm} \label{noncompactthm}
Let $(M,g)$ be a non-compact Gauduchon manifold satisfying the Assumption 1 and the Assumption 2, $(E,D)$ be a projectively flat vector bundle equipped with a fixed Hermitian metric $H_0$ over $M$. Assume that $H_0$ satisfies $$\sup\limits_{M}|\sqrt{-1}\Lambda_{\omega } G_{H_{0}}|_{H_0}<+\infty.$$ Then,
for any $\varepsilon>0$, there exists a Hermitian metric $H$ such that
$$\sqrt{-1}\Lambda_{\omega } G_{H}-\lambda \cdot \mathrm{Id}_E-\varepsilon \log (H_0^{-1}H)=0,$$
with the following estimates:
\begin{equation} \label{eq0404}
\sup_{x\in M}|\log H_0^{-1}H|_{H_0}(x)\leq \frac{1}{\varepsilon}\sup_{M}|\Phi(H_0)|_{H_0}
\end{equation}
and
\begin{equation} \label{eq0405}
\|D(\log H^{-1}_0H)\|_{L^2}\leq C(\varepsilon^{-1},\Phi(H_0),\mathrm{Vol}(M)).
\end{equation}
Moreover, if $H_{0}$ satisfies (\ref{trace31}),
then $\mathrm{tr} \log (H^{-1}_0H)=0$ and $H$ also satisfies (\ref{trace31}).
\end{thm}

\section{Proof of the theorems}

Let $(M,g)$ be a non-compact Gauduchon manifold satisfying the Assumptions 1,2,3, and $|\mathrm{d}\omega^{n-1}|_{g}\in L^2(M)$. Consider a projectively flat vector bundle $(E,D)$ over $M$ equipped with a fixed background Hermitian metric $H_0$ such that $\sup\limits_M|\Lambda_{\omega } G_{H_0}|_{H_0}<+\infty$. According to \cite[Proposition 4.3]{cjzpzz}, we can solve the following Poisson equation on $(M,g)$:
\begin{equation*} \label{lastsection41}
\sqrt{-1}\Lambda_{\omega } \partial\bar{\partial} f=
\frac{1}{r}\textmd{tr}(\sqrt{-1}\Lambda_{\omega } G_{H_0}-\lambda_{H_{0},\omega}\cdot \textmd{Id}_E),
\end{equation*}
where
$$
\lambda_{H_{0},\omega}=\frac{\int_M \sqrt{-1}\textmd{tr}(\Lambda_{\omega } G_{H_{0}})
\frac{\omega^{n}}{n!}}{\textmd{rank}(E)\textmd{Vol}(M)}.
$$

By taking a conformal transformation $\widehat{H_{0}}=e^{-2f}H_{0}$, we can verify that $\widehat{H_{0}}$ satisfies
\begin{equation} \label{lastsection113}
\textmd{tr}( \sqrt{-1}\Lambda_{\omega } G_{\widehat{H_{0}}}-\lambda_{H_{0},\omega} \cdot \textmd{Id}_E)=0.
\end{equation}
Thus, we may assume that the fixed background metric $H_{0}$ satisfies (\ref{lastsection113}).
By applying Theorem \ref{noncompactthm}, the following perturbed equation can be solved:
\begin{equation} \label{lastsectionww1}
\sqrt{-1}\Lambda_{\omega } G_{H_{\varepsilon }}
-\lambda_{H_{0},\omega} \cdot \textmd{Id}_E-\varepsilon \log h_{\varepsilon }=0,
\end{equation}
where $h_{\varepsilon }=H_{0}^{-1}H_{\varepsilon }=e^{s_{\varepsilon }}$. Since the initial Hermitian metric $H_{0}$ satisfies (\ref{lastsection113}), we obtain
\begin{equation*}\label{trace4}
\log \det (h_{\varepsilon })=\mathrm{tr}(s_{\varepsilon})=0
\end{equation*}
 and
\begin{equation*}
\textmd{tr}( \sqrt{-1}\Lambda_{\omega } G_{H_{\varepsilon }}
-\lambda_{H_{0},\omega} \cdot \textmd{Id}_E)=0.
\end{equation*}

To proceed, we need to establish the following lemmas.
\begin{lem}  \label{C0ofdistance}
 Let $(M, g)$ be a non-compact Gauduchon manifold satisfying the Assumptions 1,2,3, $(E,D)$ be a projectively flat vector bundle equipped with a fixed background Hermitian metric $H_0$
over $M$, and $|\mathrm{d}\omega^{n-1}|_{g}\in L^2(M)$. Suppose $H_{\varepsilon}$ is a solution of the perturbed equation (\ref{eq}). Define $h_{\varepsilon}=H_0^{-1}H_{\varepsilon}={e}^{s_{\varepsilon}}$, then we have
\begin{equation} \label{seq1}
-4 \int_M  {\rm tr} (\Phi(H_{0})s_{\varepsilon})\frac{\omega^n}{n!}+\int_{M}\langle \Psi(s_{\varepsilon})D(s_{\varepsilon}),D(s_{\varepsilon})\rangle_{H_{0},\omega}\frac{\omega^n}{n!}=-4\varepsilon\|s_{\varepsilon}\|^{2}_{L^2(M)}.
\end{equation}
\end{lem}

\begin{proof}
A direct application of Proposition \ref{idbundle01} and Theorem \ref{noncompactthm} yields (\ref{seq1}).
\end{proof}

\begin{lem} \label{lastsectionlemma}
Suppose that $(M,g)$ is a non-compact Gauduchon manifold as before, then there exists positive constants $C_1$ and $C_2$ independent of $\varepsilon$ such that
\begin{equation}\label{L1}\sup_{M}|s_{\varepsilon}|\leq  C_1\| s_{\varepsilon}\|_{L^2(M)}+C_2.\end{equation}
\end{lem}

\begin{proof}
It is easy to check
\begin{equation*}
\log (\frac{1}{2r}(\mathrm{tr} h_{\varepsilon} + \mathrm{tr} h^{-1}_{\varepsilon}))\leq | s_{\varepsilon }|\leq r^{\frac{1}{2}}\log (\mathrm{tr} h_{\varepsilon} + \mathrm{tr} h^{-1}_{\varepsilon}).
\end{equation*}
According to (\ref{dec}) and (\ref{eq0404}), we can obtain $|\sqrt{-1}\Lambda_{\omega } G_{H_{\varepsilon}}-\lambda \cdot \mathrm{Id}_E|_{H_{\varepsilon}}$ is uniformly bounded.
By Proposition \ref{P2}, and Assumption 3,
we have (\ref{L1}).
\end{proof}

{\bf Proof of Theorem \ref{theorem1}}

Suppose $(E,D)$ is simple, we aim to show that $H_{\varepsilon}$ converge  in the $C_{loc }^{\infty}$-topology to a Hermitian-Poisson metric $H$ as $\varepsilon \rightarrow 0$ via selecting a suitable subsequence. By employing the standard elliptic estimates and the local $C^{1}$-estimates in Proposition \ref{noncompactc1}, it suffices to obtain uniform $C^{0}$-estimate.
According to Lemma \ref{lastsectionlemma}, the crucial step is to derive a uniform $L^{2}$-estimate for $\log h_{\varepsilon }$, i.e., to prove the existence of a constant $C_{3}>0$, independent of $\varepsilon$, such that
\begin{equation}\label{L102}
\|\log h_{\varepsilon }\|_{L^2}= \left(\int_{M}|\log h_{\varepsilon }|^{2}_{H_{\varepsilon}}\frac{\omega^{n}}{n!}\right)^{\frac{1}{2}} \leq C_{3}\quad \text{for }  0<\varepsilon \leq 1. \   \
\end{equation}
We will prove (\ref{L102}) by contradiction. If not, there exists a subsequence $\varepsilon_{j}\rightarrow 0$ such that
\begin{equation*}
    \|\log h_{\varepsilon_{j}} \|_{L^2} \to +\infty \quad \text{as } \   \ j \to +\infty.
\end{equation*}
Denote
$$ s_{\varepsilon_j} = \log h_{\varepsilon_{j}},\   \ l_{j} =  \|s_{\varepsilon_j}\|_{L^2},\    \ \ u_{j}=\frac{s_{\varepsilon_j}}{l_{j}},$$
then we have
\begin{equation*}
\textmd{tr} (u_{j})=0, \ \| u_{j}\|_{L^2}=1.
\end{equation*}
From Lemma \ref{lastsectionlemma}, we obtain
\begin{equation} \label{uc0}
\sup\limits_M|u_{j}|\leq \frac{1}{\ l_{j}}(C_4\ l_{j}+C_5)<C_{6}<+\infty.
\end{equation}

$\bullet$ $\textbf{Step 1}$ \ \ We will demonstrate that $\| u_{j}\|_{L^2_1}$ are uniformly bounded. Since $\| u_{j}\|_{L^2}=1$,  it remains to derive that $\| Du_{j}\|_{L^2}$ are uniformly bounded.

By combining Proposition \ref{idbundle01} and Theorem \ref{noncompactthm}, for every $\varepsilon_{j}$, we obtain
\begin{equation} \label{seqqq212}
-4 \int_M  {\rm tr} (\Phi(H_0)u_{j})\frac{\omega^n}{n!}+l_{j}\int_{M}\langle \Psi(l_{j}u_{j})D(u_{j}),D(u_{j})\rangle_{H_0,\omega}\frac{\omega^n}{n!}=-4\varepsilon\|s_{\varepsilon_j}\|^{2}_{L^2}.
\end{equation}
Consider the function
\begin{equation*}
t\Psi(tx,ty)=
\begin{cases}
\ \ \ \ t,\ \ &x=y;\\
\frac{e^{t(y-x)}-1}{y-x},\ \ &x\neq y.
\end{cases}
\end{equation*}
By \eqref{uc0}, we can assume that $(x,y)\in [-C_{6},C_{6}]\times[-C_{6},C_{6}]$. Observed that
\begin{equation} \label{seq2}
t\Psi(tx,ty)\rightarrow
\begin{cases}
 \ \ \ \frac{1}{x-y},\ \ \ &x>y;\\
\ \ +\infty,\ \ \ &x\leq y,
\end{cases}
\end{equation}
increases monotonically as $t\rightarrow +\infty$. Let $\chi \in C^{\infty} (\mathbb{R}^{2}, \mathbb{R}^+)$ be a smooth function satisfying $\chi(x,y)<\frac{1}{x-y}$ whenever $x>y$. Combining \eqref{seqqq212}, \eqref{seq2} and the arguments in \cite[Lemma 5.4]{SIM}, we derive the following inequality:
\begin{equation} \label{seq3}
-4 \int_M  {\rm tr} (\Phi(H_0)u_{j})\frac{\omega^n}{n!}+\int_{M}\langle \chi(u_{j})D(u_{j}),D(u_{j})\rangle_{H_0,\omega}\frac{\omega^n}{n!}\leq 0, \ \ j\gg 0.
\end{equation}
By choosing $\chi(x,y)=\frac{1}{3C_{6}}$, then we obtain
$$-4 \int_M  {\rm tr} (\Phi(H_0)u_{j})\frac{\omega^n}{n!}+\frac{1}{3C_{6}}\int_{M}| D(u_{j})|^2_{H_0,\omega} \frac{\omega^n}{n!}\leq 0, \ \ j\gg 0.$$
This implies that
$$
\int_{M}| D(u_{j})|^2_{H_0,\omega} \frac{\omega^n}{n!}\leq 12C^2_{6}\sup\limits_M|\Phi(H_0)|_{H_0}\textmd{Vol}(M)<+\infty.
$$
Therefore, the subsequence $u_{j}$ are bounded in $L_1^2$. By the weak compactness of $L_1^2$, we may select a subsequence $\{u_{j_{k}}\}$ converging weakly to some $u_{\infty}$ in $L^2_1$. For simplicity, we still denote this subsequence as $\{u_{j}\}^{\infty}_{j=1}$. As $L_1^2\hookrightarrow L^2,$ we obtain
 $$1=\int_M | u_{j}|^2_{H_0}\frac{\omega^n}{n!}\rightarrow \int_M | u_{\infty}|^2_{H_0}\frac{\omega^n}{n!}, \      \   as  \    \ j \rightarrow +\infty$$
This yields that $\| u_{\infty}\|_{L^2}=1$ and $u_{\infty}$ is non-trivial.
Applying inequality \eqref{seq3} and adapting the argument in \cite[Lemma 5.4]{SIM}, we further deduce
\begin{equation} \label{seq4}
-4 \int_M \textmd{tr}(\Phi(H_0)u_{\infty})\frac{\omega^n}{n!}+\int_{M}\langle \chi(u_{\infty})(Du_{\infty}),Du_{\infty}\rangle_{H_0}\frac{\omega^n}{n!}\leq 0.
\end{equation}

$\bullet$ $\textbf{Step 2}$ \ \ We will construct a $D$-invariant sub-bundle that contradicts the simplicity of $(E,D)$.

From \eqref{seq4} and an adaptation of \cite[Lemma 5.5]{SIM}, we deduce that the eigenvalues $\{\lambda_{\alpha}\}_{\alpha=1}^{m}$ of $u_{\infty}$ are constant almost everywhere. Let $\lambda_1<\lambda_2<\cdots<\lambda_m$ represent the distinct eigenvalues of $u_{\infty}$. Since $\textmd{tr}(u_{\infty})=0$ and $\|u_{\infty}\|_{L^2}=1$, we obtain $2\leq m\leq r$. For each $\lambda_{\alpha} (1\leq \alpha\leq m-1)$, we define a function $P_{\alpha}: \mathbb{R}\rightarrow \mathbb{R}$ such that
$$
P_{\alpha}=
\begin{cases}
1,\ \ \ x\leq \lambda_{\alpha};\\
0,\ \ \ x\geq \lambda_{\alpha+1}.
\end{cases}
$$
Set $\pi_{\alpha}=P_{\alpha}(u_{\infty})$. By \cite[p.887]{SIM}, we obtain

 $(i) \pi_{\alpha}\in L^2_1;$

 $(ii)\pi_{\alpha}^2=\pi_{\alpha}=\pi_{\alpha}^{*_{H_{0}}};$

 $(iii) (\textmd{Id}_E-\pi_{\alpha})D\pi_{\alpha}=0.$

Following \cite{ppczhxz}, each $E_{\alpha}=\pi_{\alpha}(E)$ defines a $D$-invariant sub-bundle of $E$.
Thus, the projections $\{\pi_{\alpha}\}_{\alpha=1}^{m-1}$ define $m-1$ $D$-invariant sub-bundles of $E$. However, this contradicts the simplicity of $(E, D)$.

Furthermore, given that $\psi_{{H_{0}}} \in L^2$, it follows that $\psi_H \in L^2$ via the following formula: $$\psi_H=h^{-1}\circ\psi_{{H_{0}}}\circ h-\frac{1}{2}h^{-1}\circ Dh.$$

Conversely, we will demonstrate that the existence of Hermitian-Poisson metrics implies the semi-simplicity of $E$. To see this, suppose that $(E, D)$ is not simple, and let $S$ be a $D$-invariant sub-bundle. Then there exists an exact sequence
\begin{equation}
0\rightarrow S\rightarrow E \rightarrow Q \rightarrow 0.
\end{equation}
Note that $D_{S}$ and $D_{Q}$  are the connections on $S$ and $Q$ induced by $D$.
Any Hermitian metric $H$ on $E$ induces a smooth isomorphism $f_{H}:S\oplus Q\rightarrow E$. Let
\begin{equation}
    f^{*}_{H}(D)=\left(\begin{split}
    &D_{S}\ &\beta\\
    &0\ &D_{Q}
    \end{split}\right),
    \ \ \ \ f^{*}_{H}(H)=\left(\begin{split}
    &H_{S}\ &0\\
    &0\ &H_{Q}
    \end{split}\right),
\end{equation}
where $\beta\in \Omega^{1}(M,Q^{*}\otimes S)$, $H_{S}$ and $H_{Q}$ are the metrics on $S$ and $Q$ induced by $H$. Then, we have
\begin{equation}
    f^{*}_{H}(D_{H})=\left(\begin{split}
    &D_{S,H_{S}}\ &\frac{1}{2}\beta\\
    &-\frac{1}{2}\beta^{*}\ &D_{Q,H_{Q}}
    \end{split}\right),\ \ \ \ f^{*}_{H}(\psi_{H})=\left(\begin{split}
    &\psi_{S,H_{S}}\ &\frac{1}{2}\beta\\
    &\frac{1}{2}\beta^{*}\ &\psi_{Q,H_{Q}}
    \end{split}\right),
\end{equation}
\begin{equation}
    f^{*}_{H}(D^{''}_{H})=f^{*}_{H}(D^{0,1}_{H})+f^{*}_{H}(\psi^{1,0}_{H})=
    \left(\begin{split}
    &D^{''}_{S,H_{S}}\ &\frac{1}{2}\beta\\
    &\frac{1}{2}(\beta^{c})^{*}\ &D^{''}_{Q,H_{Q}}
    \end{split}\right),
\end{equation}
and
\begin{align}
   f^{*}_{H}(G_{H}) &= f^{*}_{H}(D^{''}_{H}) \circ f^{*}_{H}(D^{''}_{H}) \nonumber \\
    &= \begin{pmatrix}
        G_{S,H_{S}} + \frac{1}{4}\beta \wedge (\beta^{c})^{*} & \frac{1}{2}D^{''}_{S,H_{S}} \circ \beta + \frac{1}{2}\beta \circ D^{''}_{Q,H_{Q}} \\
        \frac{1}{2}(\beta^{c})^{*} \circ D^{''}_{S,H_{S}} + \frac{1}{2}D^{''}_{Q,H_{Q}} \circ (\beta^{c})^{*} & G_{Q,H_{Q}} + \frac{1}{4}(\beta^{c})^{*} \wedge \beta
    \end{pmatrix},
\end{align}
where $\beta^{c}=\beta^{0,1}-\beta^{1,0}$, $\beta^{*}$ is the adjoint of $\beta$ with respect to the metrics $H_{S}$ and $H_{Q}$.
Since $H$ is a Hermitian-Poisson metric, we have
\begin{equation}\label{eewwww23}
\lambda_{H,\omega}\cdot\mathrm{Id}_S=\sqrt{-1}\Lambda_{\omega} G_{S,H_{S}}+\frac{1}{4}\sqrt{-1}\Lambda_{\omega}\beta\wedge(\beta^{c})^{*}.
\end{equation}
Suppose $\omega$ is a balanced metric and $\psi_{H} \in L^2$. Then we can derive
\begin{equation}
\begin{split}
\int_{M} \sqrt{-1}\tr (\Lambda_{\omega} G_{H})\frac{\omega^{n}}{n!}&=\int_{M} \sqrt{-1} \bar{\partial} \tr(\psi_{H}^{1,0})\wedge\frac{\omega^{n-1}}{(n-1)!}\\
&=\int_{M} \sqrt{-1} \bar{\partial} \{\tr(\psi_{H}^{1,0})\wedge\frac{\omega^{n-1}}{(n-1)!}\}+\int_{M} \sqrt{-1} \tr\psi_{H}^{1,0}\wedge \frac{\bar{\partial}\omega^{n-1}}{(n-1)!}\\
&=0
\end{split}
\end{equation}
and
\begin{equation}
\int_{M} \sqrt{-1}\tr (\Lambda_{\omega} G_{S,H_{S}})\frac{\omega^{n}}{n!}=0.
\end{equation}
By taking the trace and integrating both sides of \eqref{eewwww23}, we obtain
\begin{equation}\label{eee}
\frac{1}{4}\int_{M}|\beta|^{2}\frac{\omega^{n}}{n!}=0.
\end{equation}
Thus, we have
$$\beta=0.$$
This implies that $Q$ is a $D$-invariant sub-bundle of $E$,\ i.e., $(E,D)\simeq (S,D_{S})\oplus(Q,D_{Q})$. Consequently, we conclude that $(E,D)$ is semi-simple by induction.

\hfill $\Box$ \\

\vskip 1 true cm

\bigskip
\bigskip
\bigskip

\noindent {\footnotesize {\
Jie Geng, School of Mathematics and Statistics, Nanjing University of Science and Technology, Nanjing 210094, People's Republic of China } \\
{\it Email address: gengjie1120@njust.edu.cn  }\\

\noindent {\footnotesize {\
Zhenghan Shen, School of Mathematics and Statistics, Nanjing University of Science and Technology, Nanjing 210094, People's Republic of China } \\
{\it Email address: mathszh@njust.edu.cn  }\\

\noindent {\footnotesize {\
Xi Zhang, School of Mathematics and Statistics, Nanjing University of Science and Technology, Nanjing 210094, People's Republic of China } \\
{\it Email address: mathzx@njust.edu.cn  }\\

\vskip 0.5 true cm

\end{document}